\documentclass[a4paper,12pt]{amsart}
\usepackage{amsmath,amscd,amsthm,amsxtra,amssymb}
\usepackage{epsfig,graphics,color,colortbl}
\usepackage{amssymb,latexsym}
\usepackage{mathrsfs}
\usepackage{breqn}
\usepackage{booktabs}
\usepackage{mathtools}
\usepackage{tikz, tikz-3dplot, pgfplots}
\usetikzlibrary[positioning,patterns]
\usepackage{multicol}
\usepackage[poly,all]{xy}
\usepackage{marginnote}
\usepackage{xspace}
\usepackage{yfonts}
\usepackage{enumerate}
\allowdisplaybreaks[3]
\numberwithin{equation}{section}

\usepackage[colorinlistoftodos]{todonotes}

\usepackage[backref=page,colorlinks=true, pdfstartview=FitV,linkcolor=blue,citecolor=red,urlcolor=blue]{hyperref}

\theoremstyle{plain}
\newtheorem{theorem}{Theorem}[section]
\newtheorem{lemma}[theorem]{Lemma}
\newtheorem{remark}[theorem]{Remark}
\newtheorem{corollary}[theorem]{Corollary}
\newtheorem{definition}[theorem]{Definition}
\newtheorem{proposition}[theorem]{Proposition}
\newtheorem{example}[theorem]{Example}

\numberwithin{equation}{section}
\newcommand{\catO}{\mathcal{O}}

\newcommand{\frakg}{\mathfrak{g}}
\newcommand{\frakh}{\mathfrak{h}}

\newcommand{\frakm}{\mathfrak{m}}
\newcommand{\bbC}{\mathbb{C}}
\newcommand{\bbR}{\mathbb{R}}
\newcommand{\bbZ}{\mathbb{Z}}

\newcommand{\calR}{\mathcal{R}}
\newcommand{\calA}{\mathcal{A}}

\newcommand{\calD}{\mathcal{D}}
\newcommand{\calF}{\mathcal{F}}

\newcommand{\calW}{\mathcal{W}}

\newcommand{\inv}{^{-1}}
\newcommand{\chk}{^{\vee}}
\newcommand{\Diag}{\mathrm{diag}}

\newcommand{\End}{\mathrm{End}}
\newcommand{\Mod}{\mathrm{mod}}
\newcommand{\Deg}{\mathrm{deg}}

\newcommand{\finA}{ \mathring{A} }
\newcommand{\fing}{ \mathring{\frakg} }

\newcommand{\bara}{ \bar{a} }
\newcommand{\bari}{ \bar{i} }
\newcommand{\barj}{ \bar{j} }
\newcommand{\barA}{ \bar{A} }
\newcommand{\baralpha}{ \bar{\alpha} }
\newcommand{\betazh}{ \beta^+}
\newcommand{\betafu}{ \beta^-}
\newcommand{\bgamma}{\boldsymbol{\gamma}}
\newcommand{\wt}{\mathrm{wt}}
\newcommand{\bfa}{\mathbf{a}}
\newcommand{\bfb}{\mathbf{b}}

\newcommand{\bfe}{\mathbf{e}}
\newcommand{\bff}{\mathbf{f}}
\newcommand{\bfs}{\mathbf{s}}
\newcommand{\bsf}{\boldsymbol{f}}
\newcommand{\bsg}{\boldsymbol{g}}
\newcommand{\bfepsilon}{\boldsymbol{\varepsilon}}
\newcommand{\bfm}{\mathbf{m}}
\newcommand{\bfazh}{\mathbf{a}^+}
\newcommand{\bfk}{\mathbf{k}}
\newcommand{\finW}{ \mathring{W} }

\newcommand{\finQ}{ \mathring{Q} }

\newcommand{\scrB}{ \mathscr{B}}

\newcommand{\scrT}{ \mathscr{T}}
\newcommand{\scrX}{ \mathscr{X}}
\newcommand{\UqA}{ U_q(A_{n-1}^{(1)}) }

\newcommand{\Uqg}{ U}

\newcommand{\tP}{ \widetilde{P} }
\newcommand{\td}{ \tilde{d} }
\newcommand{\tomega}{\tilde{\omega}}
\newcommand{\tpsi}{\tilde{\psi}}
\newcommand{\tdelta}{\tilde{\delta}}
\newcommand{\tW}{ \widetilde{W} }

\newcommand{\cket}[1]{|{#1}\rangle}
\newcommand{\Xzh}[1]{X^+_{#1}}
\newcommand{\inprod}[2]{\langle{#1, #2^\vee}\rangle}
\newcommand{\AffX}{X_N^{(r)}}
\newcommand{\AffC}{C_n^{(1)}}
\newcommand{\AffAt}{A_{2n}^{(2)}}
\newcommand{\AffDt}{D_{n+1}^{(2)}}
\begin{document}	
\title[On multiplicity-free weight modules]{On  multiplicity-free weight modules \\over quantum affine algebras}
	\author{Xingpeng Liu }
	\address{Shenzhen International Center for Mathematics, SUSTech, and School of Mathematical Science, University of Science and Technology of China, Hefei, 230026, Anhui, P. R. China}
	\email{xpliu127@ustc.edu.cn}

	\maketitle
	
\begin{abstract}
	In this note, our goal is to construct and study the  multiplicity-free weight modules of quantum affine  algebras. For this, we introduce the notion of shiftability condition with respect to a symmetrizable generalized Cartan matrix, and investigate its applications on the study of quantum affine algebra structures and the realizations of the infinite-dimensional multiplicity-free weight modules. 
	We also compute the highest $\ell$-weights of the infinite-dimensional multiplicity-free weight modules as highest $\ell$-weight modules. 
\end{abstract}
	
     \section{Introduction}
     Let $U_q(\frakg)$ be the quantum affine algebra (without derivation) associated to an affine Lie algebra $\frakg$ over $\bbC$ in which $q$ is not a root of unity. 
     In this note, we are concerned with  infinite-dimensional multiplicity-free weight representations, i.e., all of their weight subspaces are one-dimensional,  over  $U_q(\frakg)$. As we shall see, these representations are the basic representations towards to the infinite dimensional modules of quantum affine algebras.
     
     In the classical cases, the multiplicity-free weight representations over finite-dimensional simple Lie algebras, or more general, the bounded weight representations have been extensively studied in  \cite{BBL,BHL,GS,GS2}. These representations play a crucial role in the classification of simple weight modules of finite dimensional simple Lie algebras (cf. \cite{Mat}). For the quantum groups of finite type, Futorny-Hartwig-Wilson \cite{FHW} gave a classification of all infinite-dimensional irreducible multiplicity-free weight representations of type $ A_n $. Recently, the infinite-dimensional multiplicity-free weight representations of the quantum groups of types $A_n, B_n$ and $C_n$ were constructed in \cite{CGLW}.

    As an important class of multiplicity-free weight modules, the $q$-oscillator representations over  $U_q(\frakg)$ of types $A_n^{(1)}$, $ C_n^{(1)} $, $A_{2n}^{(2)}$, and $D_{n+1}^{(2)}$ have been obtained in the works of  T. Hayashi, A. Kuniba, M. Okado \cite{Hay90,K18,KO13,KO15}.  Our goal is to construct infinite-dimensional multiplicity-free weight representations of $U_q(\frakg)$ in a general way. For this, associated to each symmetrizable generalized Cartan matrix, we introduce a system of equations in a Laurent polynomial ring $\calA$ (essentially, the Cartan part of $U_q(\frakg)$) by the shift operators. We say that the corresponding generalized Cartan matrix satisfies  the \emph{shiftability condition} if the system of equations has solutions (see Subsection \ref{Subsec:ShiftCondition}). One result of this note is that an affine Cartan matrix satisfies the shiftability condition if and only if the relevant Dynkin diagram is one of the types mentioned above (see Theorem \ref{mainthm}). The solutions allow us to define $U_q(\frakg)$-module structures on $\calA$, and to relate the quantum affine algebra structures with the $n$-fold quantized oscillator algebra.   
    Our method for the construction is parallel with the earlier work concerning $U^0$-free modules \cite{CGLW}. Namely,  we can get the multiplicity-free weight modules of $U_q(\frakg)$  by applying the ``weighting'' procedure to the above modules on $\calA$. In particular, the $q$-oscillator representations can also be reconstructed. 
     
     For the study of weight representations of quantum affine algebras, the concepts of $\ell$-weights and $\ell$-weight vectors were proved especially useful, which allow one to  refine the spectral data properly in weight representations. For example, we have the classification of irreducible  finite-dimensional representations (cf. \cite{CP91,CP98}) and infinite-dimensional weight representation of quantum affine algebras in \cite{Her04, MY14} by highest $\ell$-weights (Note that their highest $\ell$-weights are determined by \emph{Drinfeld polynomials} and \emph{rational functions}, respectively). In this note, we shall compute explicitly the highest $\ell$-weight of the $q$-oscillator representations.  For the type $A_n^{(1)}$, the highest $\ell$-weights of $q$-oscillator representations also were discussed in  \cite{BGKNR,BGKNR2,KL2022}.

The paper is organized as follows. In Section \ref{Sec:Pre}, we give some necessary notations, and review two presentations of  quantum affine algebras. In Section \ref{Sec:Highestweightmodule} we recall the definition of highest $\ell$-weight representations. Then we obtain the classification of highest $\ell$-modules with finite weight  multiplicities in general. In Section \ref{Sec:Shiftability}, we introduce the notion of shiftability condition, and present the solutions to the corresponding system of equations, which allow us to study the compatible structures of quantum affine algebras with the $n$-fold quantized oscillator algebra. In Section \ref{Sec:construction} the infinite-dimensional multiplicity-free weight modules are constructed. In Section \ref{Sec: highestlweights}, we  compute the highest $\ell$-weight of the $q$-oscillator representations.

 \textsc{Conventions}. Let $\bbZ,\bbR$, and $ \bbC$ be the sets of integers, real numbers and complex numbers respectively, denote $\bbC \setminus \{0\}$ by $\bbC^\times$, the set of nonnegative integers by $\bbZ_{\geq 0}$, and the notation $\delta_{ij}$ stands for the Kronecker symbol in this paper.

     \section{Preliminaries and notations}\label{Sec:Pre}
     First, let us recall some necessary notations and two presentations of quantum affine algebras based on \cite{BN04,Drin87,Kac}.
     \subsection{Affine Kac-Moody algebras}
     Let $\frakg= \frakg(\AffX)$ be an affine Kac-Moody algebra with respect to the generalized Cartan matrix  $A = (a_{ij})_{i, j \in I}$ of type $\AffX$ where $I = \{0,1,\cdots,n \}$ is an indexed set and $\AffX$ is a Dynkin diagram from Table Aff $r$ of \cite{Kac}, except in the case of $\AffX = \AffAt ( n \geq 1)$, where we reverse the numbering of the simple roots. 
     
     Let $\{\alpha_i \}_{i \in I} \subset \frakh^*$ (resp. $\{\alpha_i\chk \}_{i \in I} \subset \frakh$) denote the set of simple roots (resp. simple coroots) such that $\inprod{\alpha_j}{\alpha_i} = a_{ij}$. Let $Q = \oplus_{i \in I} \bbZ \alpha_i$ be the \textit{root lattice} of $\frakg$. Set $Q_+ =  \oplus_{i \in I} \bbZ_{\geq 0} \alpha_i$. Assume that  $\delta=\sum a_i \alpha_i$ and $c = \sum a_i^\vee \alpha_i^\vee $  are the smallest positive imaginary root and a central element of  $\frakg$, where $a_i$ and $a_i\chk$ are the numerical labels of the Dynkin diagrams of $\AffX$ and its dual, respectively. 
     Let $\{\omega_i\}_{i \in I}$ denote the \textit{fundamental weights} of $\frakg$, i.e., $\inprod{\omega_i}{\alpha_j}=\delta_{ij}$ for $i, j \in I$.

     Let $W$ be the \textit{affine Weyl group} of $\frakg$ (which is a subgroup of the general linear group of $\frakh^*$) generated by the \textit{simple reflections} $s_i(\lambda) = \lambda - \inprod{\lambda}{\alpha_i}\alpha_i, \lambda \in \frakh^*, i \in I$. Note that $w(\delta) = \delta$ for all $w \in W$. Set $I_0 = I \setminus \{0 \}$. Denote by $\finW$ the subgroup of $W$ generated by the simple reflections $s_i$ for $i \in I_0$. It is a finite group. 
     
     Take the nondegenerate symmetric bilinear form $(\cdot ,\cdot )$ on $\frakh^*$ invariant under the action of $W$, which is normalized uniquely by $(\lambda, \delta)=\langle \lambda, c \rangle  $ for $\lambda \in \frakh^*$. Define $D$ as the diagonal matrix $\Diag(d_0, \cdots, d_n)$ with $d_i = a_i\inv a_i^\vee$. Then $(\alpha_i, \alpha_j) = d_ia_{ij}$ for all $i, j \in I$.
       Let $\triangle$ be the \textit{root system} of $\frakg$, $\triangle^{\pm} = \triangle \cap (\pm Q_+)$ and let $\triangle^{\mathrm{re}} = \triangle \setminus \bbZ\delta$ be the set of \textit{real roots}.  For each $\alpha \in \triangle^{\mathrm{re}}$ we set $\tilde{d}_\alpha = \mathrm{max}(1, (\alpha, \alpha)/2)$. In particular, write $\tilde{d}_i$ simply for $ \tilde{d}_{\alpha_i} $. Then 
      \begin{equation*}
      	 \tilde{d}_i = \left\lbrace 
      	\begin{array}{ll}
      		1, \quad &\text{ if } r=1 \text{ or } \AffX = \AffAt, \\
      		d_i, \quad &\text{otherwise}.
      	\end{array}
      	\right. 
      \end{equation*} 
     
     Denote by $\mathring{A} = (a_{ij})_{i,j \in I_0}$ the Cartan matrix of finite type, and let $\fing$ be the associated simple finite-dimensional Lie algebra. Then $\{\alpha_i \}_{i \in I_0}$ is a set of simple roots for $\fing$. Let $\finQ = \oplus_{i\in I_0}\bbZ \alpha_i$ be the root lattice for $\fing$, $\widetilde{P}$ the weight lattice of the euclidean space $\bbR\otimes_\bbZ \finQ \subset \frakh^*$ defined as $\tP= \oplus_{i \in I_0}\bbZ \tomega_i$, where $(\tomega_i, \alpha_j) = \delta_{ij}\tilde{d}_i$.  Then $\finQ$ can be naturally embedded into $\tP$, which provides a $W$-invariant action on $\frakh^*$ by $x(\lambda) = \lambda - (x, \lambda)\delta$ for $x \in \tP, \lambda \in \frakh^*$. 
    
    Define the \textit{extended Weyl group} by $\tW = \finW \ltimes \tP$. We also have $\tW = W \ltimes \scrT$, where $\scrT = \{w \in \tW | w(\triangle^+) \subset \triangle^+ \}$, which is a subgroup of the group of the Dynkin diagram automorphisms. An expression for $w \in \tW$ is called \textit{reduced} if $w = \tau s_{i_1}\cdots s_{i_l}$, where $\tau \in \scrT$ and $l$ is minimal. We call the minimal integer $l$ the \textit{length} of $w$, and denote it by $l(w)$. 
  
         \subsection{Quantum affine algebras}  
     	The \textit{quantum  affine algebra} $U_q(\frakg)$ in the Drinfeld-Jimbo realization \cite{Drin85,Jim85} is the unital associative algebra over $\bbC$ generated by $X_i^+$, $X_i^-$, $K_i^{\pm1}$, $i \in I$ with the following relations:
     	\begin{align}
     		\label{Relations:KK}&K_iK_i\inv = K_i\inv K_i = 1, \quad K_iK_j = K_jK_i, \\
     		\label{Relations:commutation}& K_iX_j^\pm K_i\inv = q_i^{\pm a_{ij}}X_j^{\pm}, \\
     		\label{Relations:XX}&X_i^+X_j^--X_j^-X_i^+= \delta_{ij}\frac{K_i-K_i\inv}{q_i - q_i\inv},  \\
     			\label{Relations:XX+}& \sum_{k = 0}^{1-a_{ij}}(-1)^k{1-a_{ij} \brack k}_{q_i}(X_i^\pm)^{k} X_j^\pm (X_i^\pm)^{1- a_{ij}-k} = 0,  \quad \text{ for } i \neq j,
     	\end{align}
     	where $q \in \bbC^\times$ is not a root of unity and $q_i = q^{d_i}$.  Here we have used the standard notations:  
     	\[ [m]_q = \frac{q^m-q^{-m}}{q-q\inv}, \quad [m]_q^! = [m]_q[m-1]_q\cdots [1]_q, \quad	{m \brack r}_{q} = \frac{[m]^!_q}{[r]^!_q[m-r]^!_q}. \]
     	In particular, denote $[m]_{q_i}$ by $[m]_i$ for simplicity.  
     	
     	\vspace{.3cm}

     		Let  $U^0$ be the commutative subalgebra of $U:=U_q(\frakg)$ generated by $K_i, K_i\inv, i \in I$. It is clear that each element in $U^0$ is a linear combination of the  monomials $ K_\beta:= K_0^{b_0}K_1^{b_1}\cdots K_n^{b_n} $ for $ \beta =\sum_{i\in I} b_i\alpha_i \in Q. $ 
     	In particular, $K_\delta$ is a central element in $U$. Let $U^+$ (resp. $U^-$) denote the span of monomials in $\Xzh{i}$ (resp. $X_i^-$). Recall that $U$ has a canonical triangular decomposition $U \cong U^-\otimes U^0\otimes U^+$. For later use, we note that $U^+$ is graded by $Q_+$ in the usually way: $U^+ = \oplus_{\beta\in Q_+}U^+_\beta$.  
     	
     	Let us recall the Hopf algebra structure of  $\Uqg$ with the coproduct $\Delta$, the antipode $S$, the counit $\epsilon$ defined as follows:
     	\begin{align*}  		
     		& \Delta(K_i)=K_i\otimes K_i, \quad \Delta(X_i^+)=X_i^+\otimes 1+K_i\otimes X_i^+, \\
     		& \Delta(X_i^-)=X_i^-\otimes K_i\inv + 1\otimes X_i^-, \\
     		&S(X_i^+)=-K_i\inv X_i^+, \quad S(X_i^-)=-X_i^- K_i, \quad S(K_i)=K_i\inv, \\
     		&\epsilon(X_i^+)=0=\epsilon(X_i^-), \quad \epsilon(K_i)=1.
     	\end{align*}

     	\vspace{.3cm}
     	There exists another presentation of $\Uqg$ due to Drinfeld \cite{Drin87}. Just like the realizations of the affine Kac-Moody algebras $\frakg$ as (twisted) loop algebras, this presentation of $\Uqg$ is generated by the Drinfeld's ``loop-like'' generators. 
     	
     	Consider the root datum $(X_N, \sigma)$ with $\sigma$ a diagram automorphism of $X_N$ of order $r$. Let $\barA = (\bara_{ij})_{1 \leq i, j \leq N}$ be the Cartan matrix of the type $X_N$, and let $\omega$ be a fixed primitive $r$-th root of unity. Note that if $r = 1$ (i.e., $\sigma$ is an identity) we have $N = n$, $\barA = \finA$; if $r > 1$, then $X_N$ is one of the simply laced  types: $A_N (N \geq 2)$, $D_{n+1} (n \geq 2)$, $E_6$. We use $\bari \in I_0$ to stand for one representative of the $\sigma$-orbit of $i$ on $\{1, 2, \cdots, N\}$ such that  $\bari \leq \sigma^s(i)$ for any $s$.  
     Take the set of simple roots $\{\baralpha_i \}_{1 \leq i \leq N}$  and the normalized bilinear form $( , )$  (by abuse of notation) such that $(\baralpha_i, \baralpha_j) = d_i a_{ij}$ if $r=1$, otherwise $(\baralpha_i, \baralpha_j) =  \bara_{ij}$ for $1 \leq i, j \leq N$.  
     
      The quantum affine algebra $\Uqg$ (add the central elements $K_\delta^{\pm 1/2}$) is isomorphic to the algebra generated by $x_{i, k}^\pm (1 \leq i \leq N, k \in \bbZ)$, $h_{i,k} (1 \leq i \leq N, k \in \bbZ\setminus\{0\})$, $k_i^{\pm 1} (1 \leq i \leq N)$ and the central elements $C^{\pm 1/2}$, subject to the following relations:
     \begin{align}
     	\nonumber&x_{\sigma(i),k}^\pm = \omega^kx_{i,k}^\pm, h_{\sigma(i),k}^\pm = \omega^kh_{i,k}^\pm, k_{\sigma(i)}^{\pm 1} = k_i^{\pm 1}, \\
     	\nonumber&k_ik_i\inv = k_i\inv k_i = 1, k_ik_j = k_jk_i, k_ih_{j,l}= h_{j,l}k_i,\\
     	\nonumber&k_ix_{j,k}^\pm = {q_{\bari}}^{\pm a_{\bari\barj}}x_{j,k}^\pm k_i, \\
     	\label{DrinfeldRelations}&[h_{i,k}, h_{j,l}] = \delta_{k,-l}\frac{1}{k} \Big(\sum_{s=1}^r [\frac{k(\baralpha_i, \baralpha_{\sigma^s(j)})}{d_\bari}]_\bari \;\omega^{ks} \Big) \frac{C^k-C^{-k}}{q_\bari - {q_\bari}\inv},\\
     	\nonumber&[h_{i,k}, x_{j,l}^\pm]=\pm \frac{1}{k}\Big(\sum_{s=1}^r [\frac{k(\baralpha_i, \baralpha_{\sigma^s(j)})}{d_\bari}]_\bari \; \omega^{ks} \Big) C^{\mp |k|/2}x_{j,k+l}^\pm,\\
     	\nonumber&[x_{i,k}^+, x_{j,l}^-]=\Big(\sum_{s=1}^r\frac{ \delta_{\sigma^s(i)j}\omega^{sl}}{\tilde{d}_\bari}\Big) \frac{C^{(k-l)/2}\psi_{i, k+l}^+ - C^{-(k-l)/2}\psi_{i, k+l}^-}{q_\bari - {q_\bari}\inv},  
     \end{align}
     where $\psi_{i,k}^\pm$'s are the elements determined by the following identity of the formal power series in $z$: 
 \begin{align}\label{HHseries}
 	  \sum_{k=0}^{\infty}\psi_{i,\pm k}^\pm z ^{\pm k} = k_i^{\pm 1} \mathrm{exp}\Big(\pm (q_\bari - {q_\bari}\inv) \sum_{l=1}^\infty h_{i, \pm l}z^{\pm l}\Big),    
 \end{align}
     together with the \textit{quantum Serre-Drinfeld relations}, whose explicit forms will be not  used in this paper. One can refer to \cite{Drin87} for more details and  to \cite{Beck94,Jing96} and \cite{Dami98,Dami12,Dami15}\footnote{The author used the notations ${\widetilde{H}}^{\pm}_{i,l}$, $H_{i,l}$, which are related with $ \psi_{i,l}^{\pm} $, $ h_{i,l} $ defined in this note by  $ {\widetilde{H}}^{\pm}_{i,l}=C^{l/2}k_i^{\mp1}\psi_{i,l}^{\pm}$ and $H_{i,l} = C^{l/2}h_{i,l}$.} for a proof. 
     
     Under the isomorphism, we have $X_i^\pm = x_{i,0}^\pm$, $K_i^{\pm1} = k_i^{\pm1}$ for $i \in I_0$, and $K_\delta = C$. 
     Note that $\psi_{i, -k}^+ = \psi_{i,k}^- =0$ for any positive integers $k$, and $\psi_{i,0}^\pm = k_i^{\pm 1}$ from the identity (\ref{HHseries}).
     
     From the relations in Drinfeld presentation, $\Uqg$ is essentially generated by the generators  $x_{i, \td_ik}^\pm (i \in I_0, k \in \bbZ)$, $h_{i,\td_ik} (i \in I_0, k \in \bbZ\setminus\{0\})$, $k_i^{\pm 1} (i \in I_0)$ and the central elements $C^{\pm 1/2}$ (see \cite[Proposition 4.25]{Dami12}). Moreover, the quantum affine algebra $U$ has a triangular decomposition \cite{CP94,CP98}:
     \begin{equation}\label{Triang}
     	\Uqg \cong  U(\leqslant) \otimes U(0) \otimes U(\geqslant)
      \end{equation}
     where $U(\geqslant)$ (resp. $U(\leqslant)$) is the subalgebra generated by $x_{i, \td_ik}^+ $ (resp. $x_{i, \td_ik}^- $), $ i \in I_0$, $k \in \bbZ$, and $U(0)$ is the subalgebra generated by $C^{\pm 1/2}$, $k_i^\pm$, $h_{i,k}, i\in I_0, k \in \bbZ\setminus\{0\}$.

     \section{Highest $\ell$-weight representations with finite weight multiplicities}  \label{Sec:Highestweightmodule}
     
     In this section, we recall basic notations of representations over quantum affine algebras: weight modules, $\ell$-weights, and highest $\ell$-weight modules. Most of the definitions and results in this section are well-known, one can refer to \cite{CP91,MY14}.

     \subsection{Highest $\ell$-weight modules} We begin with the notion of highest $\ell$-weight modules. 
     Thanks to the Hopf algebra structure of $U^0$ (inherits from $U$), the set of all \emph{algebra characters} of $U^0$, i.e., all algebra homomorphisms from $U^0$ to $\bbC$,  has an abelian group structure, the addition and the inverse are given by 
          \[(\lambda + \mu)(u) = (\lambda\otimes \mu)\circ \Delta(u), \quad (-\lambda)(u) = \lambda\circ S(u) \]
     for any algebra characters $\lambda, \mu$, and $u \in U^0$. Denote this group simply by $(\scrX, +)$. Any $\beta \in \frakh^*$ induces a character in  $\scrX$ by assigning $  K_i $ to $ q^{(\beta, \alpha_i)} $ for $i\in I$, which is unique up to a constant multiple of $\delta$, so we still denote it by $\beta\in \scrX$.

     For a $\Uqg$-module $V$ and $\lambda \in \scrX$, define 
     \[V_\lambda = \{v \in V \; |\;  u.v = \lambda(u)v, \forall u \in U^0\}.\]
    By the defining relations (\ref{Relations:commutation}) we have $ X_i^\pm.V_\lambda \subset V_{\lambda \pm \alpha_i} $. If $V_\lambda$ is nonzero, then we say $\lambda$ is a \emph{weight} of $V$, and $V_\lambda$ is a \emph{weight space} of \emph{weight} $\lambda$, a nonzero vector $v \in V_\lambda$ is called a \emph{weight vector} of weight $\lambda$. If the weight space $V_\lambda$ is finite-dimensional, then $\dim V_\lambda$ is called the \emph{multiplicity} of the weight $\lambda$. Call 
     $V$  a \emph{weight module} if $V = \oplus_\lambda V_\lambda$. Moreover, a weight module $V$ is said to be \emph{multiplicity-free} if $\dim V_\lambda \leq 1$ for all $\lambda\in \scrX$.

     \emph{Throughout this note, we assume that 
     	the central element $C$ acts trivially on a $\Uqg$-module.} So any weight $\lambda$ of  a $\Uqg$-module is \emph{level-zero}, that is, $\lambda(K_\delta) = 1$.

     Note that the actions of $\psi_{i, k}^\pm$'s on a $\Uqg$-module commute with each other by (\ref{DrinfeldRelations}) and (\ref{HHseries}). For a weight $\lambda$ of $V$ with finite multiplicity, we may refine the weight space $V_\lambda$ as 
     \[V_\lambda = \bigoplus_{\bgamma: \wt(\bgamma)=\lambda}V_\bgamma, \]
     \[V_{\bgamma} = \{v \in V_\lambda\; | \;\forall \;1 \leq i \leq N, k \geq 0, \exists m \in \bbZ_{>0}, (\psi_{i, \pm k}^\pm - \gamma_{i, \pm k}^\pm)^m.v=0 \}, \]
     where $\bgamma = (\gamma_{i, \pm k}^\pm)_{1 \leq i \leq N, k \in \bbZ_{\geq 0}}$ is any $N$-tuple of sequences of complex numbers satisfying that $\gamma_{i, 0}^+\gamma_{i, 0}^-=1$ and $\gamma_{\sigma(i), \pm k}^\pm = \omega^{\pm k}\gamma_{i, \pm k}^\pm$ for all $1 \leq i \leq N$, and we associate $\bgamma$ with a level-zero weight $\wt(\bgamma) \in \scrX$ by setting $\wt(\bgamma)(K_i) = \gamma_{i,0}^+$ for all $i \in I_0$. Call such a sequence $\bgamma$ an \emph{$\ell$-weight}, $V_\bgamma$ the \emph{$\ell$-weight space} of $\bgamma$ if $V_\bgamma$ is not zero.  

        Given an $\ell$-weight $\bgamma$. The defining relations in the Drinfeld presentation imply that $\bgamma$ is completely determined by the tuple of complex numbers   $ (\gamma_{i, \pm \td_i k}^\pm)_{i \in I_0, k \in \bbZ_{\geq 0}} $. Note that  $\gamma_{i, k}^\pm$'s for $\td_i \nmid k$ are zero. Hence we may write 
       $ \bgamma \equiv (\gamma_{i, \pm \td_i k}^\pm)_{i \in I_0, k \in \bbZ_{\geq 0}} $ directly 
         without any ambiguity. 
         
         Now we can define the highest $\ell$-weight modules.
     \begin{definition}
     We say $V$ is a \emph{highest $\ell$-weight modules} of highest $\ell$-weight $\bgamma$ if $V = \Uqg.v $ for some non-zero vector $v \in V$ such that 
     $ x_{i, k}^+.v =0$ for $ 1 \leq i \leq N, k \in \bbZ $, and $ \psi_{i, \pm k}^\pm.v = \gamma_{i, \pm k}^\pm v$ for  $ 1 \leq i \leq N, k \in \bbZ_{\geq 0} $. By $ (\ref{Triang}) $ $\dim V_\bgamma =1$, so $v$ is unique up to a scalar; we call it the \emph{highest $\ell$-weight vector} of $V$.  
 \end{definition}

     \subsection{The classification theorem: rationality}\label{Subsection:Affineosci} 
In this subsection we give the classification of simple highest $\ell$-weight modules with finite weight multiplicity, which appeared in \cite{MY14} for untwisted cases.

     	We say an $\ell$-weight $\bsf=(f_{i, \pm \td_ik}^\pm)_{i\in I_0, k \in \bbZ_{\geq 0}}$ is  \emph{rational} if there is a tuple of complex-valued rational functions $ (f_i(z))_{i\in I_0} $ in a formal variable $z$ such that for each $i\in I_0$, $f_i(z)$ is regular at $0$ and $\infty$, $f_i(0)f_i(\infty) = 1$ and 
     	\[\sum_{k=0}^\infty f_{i, \td_ik}^+z^k = f_i(z) = \sum_{k=0}^\infty f_{i, -\td_ik}^-z^{-k}  \]
     	in the sense that the left and right hand sides are the Laurent expansions of $f_i(z)$ at $0$ and $\infty$,  respectively.    
     	
     Let $\calR$ be the set of all rational $\ell$-weights. Then $\calR$ forms an abelian group with the group operation $(\bsf, \bsg) \mapsto \bsf\bsg$ being given by component-wise multiplication of the corresponding tuples of rational functions. 
     In what follows, we do not always distinguish between a rational $\ell$-weight $\bsf$ and the corresponding tuple $(f_i(z))_{i\in I_0}$ of rational functions. 
     
     Recall from \cite{CP91,CP98} that simple finite-dimensional modules of $\Uqg$ are highest $\ell$-weight modules, and their highest $\ell$-weights $\bsf$ are parametrized by the tuples of the \emph{Drinfeld polynomials}. More precisely,  there exists a tuple of polynomials $(P_i(z))_{i\in I_0}$ with all $P_i(z)$ having constant coefficient $1$ such that  $\bsf$ satisfies that  for $i \in I_0$, 
     \[ f_i(z) = \begin{dcases}
     	q_n^{2\deg P_n} \frac{P_n(q_n^{-4}z)}{P_n(z)} & \text{if}\;  (X_N^{(r)}, i) = (A_{2n}^{(2)}, n), \\
     	q_i^{\deg P_i} \frac{P_i(q_i^{-2}z)}{P_i(z)}& \text{otherwise}.
     \end{dcases}
     \]
     Therefore, the highest $\ell$-weight of any simple finite-dimensional module is rational. 
     
     In general, we have the following  theorem.
     
     \begin{theorem}
     	Let $V$ be an irreducible highest $\ell$-weight module. Then all weight spaces of $V$ are finite-dimensional if and only if its highest $\ell$-weight $\bsf$ belongs to $\calR$. 
     \end{theorem}
     \begin{proof}
     	For the non-twisted cases, one can refer to \cite[Theorem 3.7]{MY14}	and the references therein. The proof of the twisted cases is essentially parallel to that of the untwisted cases thanks to the  triangular decomposition (\ref{Triang}) of the Drinfeld realization. 
     \end{proof}

     \section{Shiftability conditions and algebra homomorphisms}\label{Sec:Shiftability}
     In this section, the notion of the shiftability condition with respect to a generalized Cartan matrix will be introduced, and the compatible structures of the quantum affine algebras with the $n$-fold $q$-oscillator algebras are given from the $q$-shiftability condition. 

 \subsection{Shiftability conditions} \label{Subsec:ShiftCondition}
  Given any symmetrizable generalized Cartan matrix $A=(a_{ij})_{i,j\in I}$. Let $\calA$ be the Laurent polynomial ring over $\bbC$ in the variables $x_i, i \in I$, i.e., 
  $  \calA = \bbC[x_i^{\pm1}, i \in I].  $ 
   For each $i \in I$, consider the algebra automorphism $\zeta_i: \calA \rightarrow \calA$ given by $\zeta_i(x_j) = q_i^{-\delta_{ij}}x_j$ for $j \in I$. For any distinct $i, j \in I$, we say  a pair of Laurent polynomials $(f, g)$ in $\calA$ is \emph{$(i, j)$-shiftable} if $f, g$ satisfy the equation
 \[fg = \zeta_j\inv (f)\zeta_i\inv(g). \]
 
 Set $\{x\}_i := \frac{x-x\inv}{q_i- q_i\inv}$ for any unit $x$ in $\calA$, and write $\{x\} = \frac{x-x\inv}{q- q\inv}$ for simplicity.   Define the elements $y_i, y_i^{-1} \in \calA$ as follows: 
 \[y_i^{\pm 1} = \prod_{j \in I}x_j^{\pm a_{ji}}. \] 
 
 Consider the following system of equations with respect to the variables $\phi_i, i \in I$ in $\calA$:
 \begin{equation}\label{mainequ}
 	\left\{
 	\begin{array}{l}
 		\zeta_i(\phi_i)-\phi_i = \{y_i\}_i, \\[6pt]	
 		\phi_i\phi_j = \zeta_j\inv (\phi_i) \zeta_i\inv (\phi_j),
 	\end{array}
 	\quad i, j \in I, i \neq j.
 	\right.
 \end{equation} 
In general, this system of equations does not always have a solution. It depends on the choice of the generalized Cartan matrix $A$. Therefore, we can say $A$ admits the \textit{$q$-shiftability condition} when the corresponding system of equations  (\ref{mainequ}) has a solution. 

By a quick computation, we obtain a family of solutions to (\ref{mainequ}) for $A$ of types $A_2$ and $A_1^{(1)}$. 

\begin{example}\label{ExampleA}
	$ \mathrm{(i)} $ For the type $A_2$, a pair of Laurent polynomials $(\phi_1, \phi_2 )$, where $\phi_1 = \{qbx_1\}\{bx_1\inv x_2 \}$ and $\phi_2= \{qbx_1\inv x_2 \}\{bx_2\inv \}$ for each scalar $b \in \bbC^\times$ is a solution; 
	\vspace{.2cm}
	
	$ \mathrm{(ii)} $ For the type $A_1^{(1)}$, consider the Laurent polynomials $\phi_0 = \{qbx_0x_1\inv \}\{bx_0\inv x_1 \}$ and $\phi_1 = \{qbx_0\inv x_1 \}\{bx_0x_1\inv \}$ for any scalar $b \in \bbC^\times$. It is easy to check that $(\phi_0, \phi_1 )$ is a solution. 
\end{example}

In what follows, the $q$-shiftability condition for the generalized Cartan matrices of affine types will be investigated. Now assume that $A$ is an affine Cartan matrix as in Section \ref{Sec:Pre}. Then we have the first main result in this section.

\begin{theorem}\label{mainthm}
	There exists an $(n+1)$-tuple of Laurent polynomials in $\calA$ satisfying the system of equations $(\ref{mainequ})$ if and only if  $A$ is  the type $A_{n}^{(1)} (n \geq 1), C_{n}^{(1)} (n \geq 2), A_{2n}^{(2)} (n \geq 1)$ or $D_{n+1}^{(2)} (n \geq 2)$. 
\end{theorem} 
The proof of  Theorem \ref{mainthm} will be given in Appendix \ref{Appendix}. Here we list all tuples of Laurent polynomials $(\phi_i)_{i \in I}$ satisfying (\ref{mainequ}) for each affine Cartan matrix $A$ in the theorem above. 
\begin{alignat*}{4}
	&A_{n}^{(1)} (n \geq 1): && \Big(\{qb_A z_0 \}\{b_A z_{1} \},  \{qb_A z_1 \}\{b_A z_{2} \}, \cdots, \{qb_A z_n \}\{b_A z_{0} \}\Big)\\
	&C_{n}^{(1)} (n \geq 2): && \Big(\{ q_0b_C z_1\inv\}_0\{ b_Cz_1\}_0, \{q_1b_C z_{1} \}_1\{b_C z_2\}_1, \\
	 & && \hspace{2.5cm} \cdots, \{q_{n-1}b_C z_{n-1} \}_{n-1}\{b_C z_{n}\}_{n-1}, \{q_nb_C z_{n}\}_n\{b_C z_{n}\inv \}_n \Big)\\
	&A_{2n}^{(2)} (n \geq 1):  && \Big( \{\imath q^{-\frac{3}{2}} z_1\}_0\{\imath q^{-\frac{1}{2}} z_{1}\}_0,  \{\imath q^{\frac{1}{2}} z_1\}_1\{\imath q^{-\frac{1}{2}} z_2\}_1,\\
	 & && \hspace{2.5cm} \cdots, \{\imath q^{\frac{1}{2}} z_{n-1}\}_{n-1}\{\imath q^{-\frac{1}{2}} z_{n}\}_{n-1}, \frac{\imath}{q_n-q_n\inv} \{\imath q^{\frac{1}{2}} z_n\}_n\Big)\\
	&D_{n+1}^{(2)} (n \geq 2): && \Big(\frac{\imath}{q_0-q_0\inv} \{\imath q\inv z_1\}_0,  \{\imath q z_1\}_1\{\imath q\inv z_2\}_1, \\
	& && \hspace{3.5cm}\cdots, 	\{\imath q z_{n-1}\}_{n-1}\{\imath q\inv z_{n}\}_{n-1}, \frac{\imath}{q_n-q_n\inv} \{\imath q z_n\}_n\Big)
\end{alignat*}
where $\imath = \sqrt{-1}$.  The elements $z_i \in \calA$ involved in the above solutions, and the relations in our notations are given as follows for each type: 
\[z_i = x_{i-1}\inv x_i, \; z_0=(z_1\cdots z_n)\inv, \; y_i = z_iz_{i+1}\inv,\; y_n=z_nz_0\inv, \;  b_{A_{n}^{(1)}}  \in \bbC^\times \; \text{in  } A_{n}^{(1)}; \]
\[z_i = x_{i-1}\inv x_{i}, \;  y_0 = z_1^{-2}, \; y_i=  z_iz_{i+1}\inv, \; y_n = z_n^2, \; b_{ C_{n}^{(1)} } =q^{-1/4} \text{ or } \imath q^{-1/4}\;  \text{in } C_{n}^{(1)};\]
\[z_i = x_{i-1}\inv x_{i},\; z_n = x_{n-1}^{-1}x_n^2, \;y_0 = z_1^{-2}, \; y_i=  z_iz_{i+1}\inv, \; y_n = z_n,  \;  b_{A_{2n}^{(2)}}=\imath q^{-\frac{1}{2}}\;  \text{in } A_{2n}^{(2)};  \]
\[ z_1 = x_0^{-2}x_1,\;  z_i = x_{i-1}\inv x_{i}, \; z_n = x_{n-1}\inv x_n^2, \; y_i=  z_iz_{i+1}\inv, \; y_n = z_n,\;   y_0 = z_1^{-1}, \; b_{D_{n+1}^{(2)}}=\imath q\inv \;\text{in } D_{n+1}^{(2)}.\]
By our convention, the Dynkin diagrams of the above four types and the corresponding $q_i=q^{d_i}$ are the following:

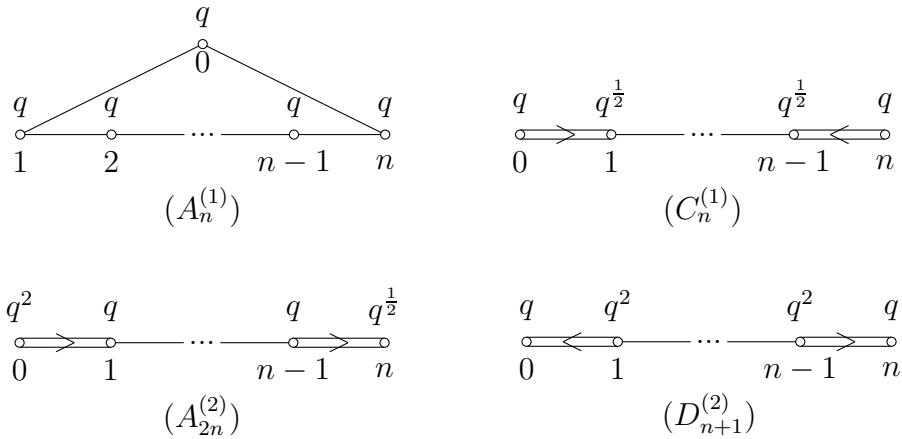
{\centering
	
	\begin{tikzpicture}[scale=1.2]
				\draw (0,0) -- (1,0);
				\draw (1,0) -- (1.8,0);
				\draw (2.2, 0) -- (3,0);
				\draw (3,0) -- (4,0);	
				\draw (0,0) -- (2, 1);
				\draw (2,1) -- (4,0);
				\draw[fill=white] (0,0) circle(.05);
				\node at (2, 0.8) {$  0  $};
				\node at (2, 1.3) { $ q $ };
				\draw[fill=white] (2, 1) circle(.05);
				\node at (0,-.3) {$ 1 $};
				\node at (0,.3) {$ q $};
				\draw[fill=white] (1,0) circle(.05);
				\node at (1,-.3) {$ 2 $};
				\node at (1,.3) {$ q $};
				\draw[fill=white] (3,0) circle(.05);
				\node at (3,-.3) {$ n-1 $};
				\node at (3,.3) {$ q $};
				\draw[fill=white] (4,0) circle(.05);
				\node at (4,-.3) {$ n $};		
				\node at (4,.3) {$ q $};		
				\node at (2,0) { ... };
				\node at (2,-.8) {$ (A^{(1)}_n)$};
			\end{tikzpicture}				
\hspace{1cm}		
			\begin{tikzpicture}[scale=1.2]
			\draw (0,0.05) -- (1,0.05);
			\draw (0,-0.05) -- (1,-0.05);
			\draw (0.6,0) -- (0.4,0.1);
			\draw (0.6,0) -- (0.4,-0.1);
			\draw (1,0) -- (1.8,0);
			\draw (2.2, 0) -- (3,0);
			\draw (3,0.05) -- (4, 0.05);
			\draw (3,-0.05) -- (4, -0.05);
			\draw (3.4,0)--(3.6, 0.1);
			\draw (3.4,-0)--(3.6, -0.1);
			\draw[fill=white] (0,0) circle(.05);
			\node at (0,-.3) {$ 0 $};
			\node at (0,.35) {$ q $};
			\draw[fill=white] (1,0) circle(.05);
			\node at (1,-.3) {$ 1 $};
			\node at (1,.4) {$ q^{\frac{1}{2}} $};
			\draw[fill=white] (3,0) circle(.05);
			\node at (3,-.3) {$ {n-1} $};
			\node at (3,.4) {$ q^{\frac{1}{2}} $};
			\draw[fill=white] (4,0) circle(.05);
			\node at (4,-.3) {$ n $};
			\node at (4,.35) {$ q $};
			\node at (2,0) {$ ... $};
			\node at (2,-0.8) {$(C_n^{(1)})$ };
		\end{tikzpicture}
	
	\vspace{.5cm}
	\begin{tikzpicture}[scale=1.2]
	\draw (0,0.05) -- (1,0.05);
	\draw (0,-0.05) -- (1,-0.05);
	\draw (0.6,0) -- (0.4,0.1);
	\draw (0.6,0) -- (0.4,-0.1);
	\draw (1,0) -- (1.8,0);
	\draw (2.2, 0) -- (3,0);
	\draw (3,0.05) -- (4, 0.05);
	\draw (3,-0.05) -- (4, -0.05);
	\draw (3.4,0.1)--(3.6, 0);
	\draw (3.4,-0.1)--(3.6, 0);
	\draw[fill=white] (0,0) circle(.05);
	\node at (0,-.3) {$ 0 $};
	\node at (0,.4) {$ q^2 $};
	\draw[fill=white] (1,0) circle(.05);
	\node at (1,-.3) {$ 1 $};
	\node at (1,.35) {$ q $};
	\draw[fill=white] (3,0) circle(.05);
	\node at (3,-.3) {$ {n-1} $};
	\node at (3,.35) {$ q $};
	\draw[fill=white] (4,0) circle(.05);
	\node at (4,-.3) {$ n $};
	\node at (4,.4) {$ q^{\frac{1}{2}} $};
	\node at (2,0) {$ ... $};
	\node at (2,-0.8) {$(A_{2n}^{(2)})$ };
\end{tikzpicture}
\hspace{1cm}
	\begin{tikzpicture}[scale=1.2]
	\draw (0,0.05) -- (1,0.05);
	\draw (0,-0.05) -- (1,-0.05);
	\draw (0.6,0.1) -- (0.4,0);
	\draw (0.6,-0.1) -- (0.4,0);
	\draw (1,0) -- (1.8,0);
	\draw (2.2, 0) -- (3,0);
	\draw (3,0.05) -- (4, 0.05);
	\draw (3,-0.05) -- (4, -0.05);
	\draw (3.4,0.1)--(3.6, 0);
	\draw (3.4,-0.1)--(3.6, 0);
	\draw[fill=white] (0,0) circle(.05);
	\node at (0,-.3) {$ 0 $};
	\node at (0,.35) {$ q $};
	\draw[fill=white] (1,0) circle(.05);
	\node at (1,-.3) {$ 1 $};
	\node at (1,.4) {$ q^{2} $};
	\draw[fill=white] (3,0) circle(.05);
	\node at (3,-.3) {$ {n-1} $};
	\node at (3,.4) {$ q^{2} $};
	\draw[fill=white] (4,0) circle(.05);
	\node at (4,-.3) {$ n $};
	\node at (4,.35) {$ q $};
	\node at (2,0) {$ ... $};
	\node at (2,-0.8) {$(D_{n+1}^{(2)})$ };
\end{tikzpicture}

}

\begin{remark}
	One  can also consider the shiftability condition for a generalized Cartan matrix $A$ in the classical sense. 
	More precisely, consider the polynomial ring $\calA^+ = \bbC[x_i, i \in I]$, and the algebra automorphisms ${\zeta}_i: \calA^+ \rightarrow \calA^+$ defined by ${\zeta}_i(x_j) = x_j - \delta_{ij}$ for all $j \in I$. Denote ${y}_i = \sum_{j \in I} a_{ij}x_j$. Then a similar system of equations in $\calA^+$ (replace $\{y_i\}_i$ in $ (\ref{mainequ})$ by $y_i$)  can be obtained. 
\end{remark}

    \subsection{Quantized oscillator algebra and algebra homomorphisms}\label{Subsec:quantumoscillator} One interesting application of the  $q$-shiftability condition is to study the compatible structures of  quantum affine algebras of types $X_N^{(r)}$ with the $n$-fold quantized oscillator algebra.

    Fix $\nu \in \bbC^\times$. The \textit{(symmetric) quantized oscillator algebra} $\scrB_\nu$ is the unital associative algebra over $\bbC$ generated by four elements $\bfazh$, $\bfa$, $\bfk^{\pm1}$ subject to the relations 
  \begin{alignat*}{4}
    & [\bfa, \bfazh]_\nu = \bfk, \quad && [\bfa, \bfazh]_{\nu\inv} = \bfk\inv, \quad && \bfk\bfk\inv = \bfk\inv\bfk =1, \\
    & \bfk \bfa \bfk\inv = \nu\inv \bfa, \quad && \bfk \bfazh\bfk\inv = \nu\bfazh, &&   	
  \end{alignat*}
where $[x, y]_\nu:=xy-\nu\inv yx$. Then we have $\bfazh\bfa = \{\bfk\}$, $\bfa\bfazh = \{\nu\bfk \}$ and $\{\bfk\} \bfazh = \bfazh \{\nu\bfk\}$, $ \bfa \{ \bfk \}=\{\nu\bfk \}\bfa$ in $\scrB_\nu$.  Here we define $\{\mathbf{x} \}=\{\mathbf{x} \}_\nu = (\mathbf{x}-\mathbf{x}\inv)/(\nu-\nu\inv)$ for $\mathbf{x} = \bfk$ or $\nu\bfk$.

One can easily check the following results.
\begin{lemma}\label{Lemma:automorphisms}
		$ \mathrm{(i)} $ There exists a unique $\bbC$-algebra automorphism (an involution) $\vartheta: \scrB_\nu \rightarrow \scrB_\nu$ such that 
	    $  \vartheta(\bfazh) =- \bfa $,  $ \vartheta(\bfa) = \bfazh $ and $   \vartheta(\bfk) = \nu\inv \bfk\inv $.\\
		$ \mathrm{(ii)} $ For any $b \in \bbC^\times$ and $m \in \bbZ$, there exists a family of $\bbC$-algebra automorphisms $\theta_{b,m}: \scrB_\nu \rightarrow \scrB_\nu$ such that 
	$ 	\theta_{b,m}(\bfa)=b\bfk^m\bfa $, $ \theta_{b,m}(\bfazh) = b\inv \bfazh\bfk^{-m} $ and $  \theta_{b,m}(\bfk^{\pm 1}) = \bfk^{\pm 1} $. 
\end{lemma}

Consider the algebra $\scrB_\nu^{\otimes n}$  of the $n$-fold tensor product of $ \scrB_\nu $. Denote the generators of its $i$-th component by $\bfazh_i$, $\bfa_i$ and $\bfk_i^{\pm 1}$, which satisfy the above relations.  Let $U_q(X_N^{(r)})$ be the quantum affine algebra $\Uqg$ of the type $X_N^{(r)}$ in Theorem \ref{mainthm}. For convenience, if $X$ is the type $A$ then we shall deal with $A_{n-1}^{(1)} (n \geq 2)$ instead of $A_{n}^{(1)} (n \geq 1)$ from now on. 

Fix a solution $(\phi_{i})_{i\in I}$ in Subsection \ref{Subsec:ShiftCondition}. We define the algebra homomorphism $\pi_{X_N^{(r)}}: U_q(X_N^{(r)}) \rightarrow \scrB_\nu^{\otimes n}$ in the following way: regard $\phi_i$ and $\zeta_i(\phi_i)$ as the images of $X_i^-X_i^+$ and $X_i^+X_i^-$ respectively under $\pi_{X_N^{(r)}}$ by setting $\bfk_i = q\inv b_A\inv z_i\inv$ for the type $A_{n-1}^{(1)}$, and $\bfk_i = \imath \nu^{-1/2}z_i\inv$ otherwise (Here we consider the solution with $b_C = \imath q^{-1/4}$ for the type $C_n^{(1)}$), where $\nu$ is defined as in the following proposition for each type. Then the relations (\ref{Relations:XX}) for $i=j$ holds under $\pi_{X_N^{(r)}}$ since $\phi_i$ satisfies $\zeta_i(\phi_i)-\phi_i = \{y_i\}_i$. In this sense, $\calA_0:=\bbC[\bfk_1^{\pm1}, \cdots, \bfk_n^{\pm1}]$ is a subalgebra of $\calA$, and $\pi_{X_N^{(r)}}(K_i) = y_i$ for $i\in I$. On the other hand, we choose $\pi_{X_N^{(r)}}(X_i^\pm)\in \scrB_\nu^{\otimes n}$ satisfying that 
\[\pi_{X_N^{(r)}}(X_i^\pm)f = \zeta_i^{\pm 1}(f)\pi_{X_N^{(r)}}(X_i^\pm) \]
for any $f \in \calA_0$. The above choice yields the relations (\ref{Relations:KK})-(\ref{Relations:XX+}) hold. 
Then we get the following algebra homomorphisms, which were obtained in \cite{Hay90,KOS15}. 

\begin{proposition}[\cite{Hay90,KOS15}]\label{Prop:alghom}
	For a parameter $z$, there exist algebra homomorphisms from $U_q(X_N^{(r)})$ to $\scrB_\nu^{\otimes n}[z, z\inv]$ defined as follows: 
	
		\noindent\emph{($ A_{n-1}^{(1)}, \nu =q $)}	\hspace{-3cm}\centerline{ $ \pi_{A_{n-1}^{(1)}, z} : U_q(A_{n-1}^{(1)}) \rightarrow \scrB_\nu^{\otimes n}[z, z\inv] $}	
	\[X^+_i \mapsto z^{\delta_{i,0}}\bfa_i\bfazh_{i+1}, \quad\quad X^-_i \mapsto z^{-\delta_{i,0}}\bfazh_i\bfa_{i+1}, \quad\quad K_i \mapsto \bfk_i\inv \bfk_{i+1}. \]
	In this type, we always read the index $i$ as $i$ modulo $ n$. 
	
		\noindent\emph{($ C_{n }^{(1)}, \nu = q^{\frac{1}{2}} $)}	\hspace{-3cm}\centerline{ $ \pi_{C_{n }^{(1)}, z} : U_q(C_{n }^{(1)} ) \rightarrow \scrB_{\nu}^{\otimes n}[z, z\inv] $}	
	 \begin{alignat*}{4}
		&X_0^+ \mapsto z(\bfazh_1)^2/[2]_\nu, \quad && X_0^- \mapsto  z^{-1}\bfa_1^2/[2]_\nu, \quad &&  K_0 \mapsto -\nu \bfk_1^2,\\
		&X_i^+ \mapsto \bfa_i\bfazh_{i+1}, \quad && X_i^- \mapsto \bfazh_i\bfa_{i+1}, \quad && K_i \mapsto \bfk_i\inv\bfk_{i+1}, \\
		& X_n^+ \mapsto \bfa_n^2/[2]_\nu, \quad && X_n^- \mapsto  (\bfazh_n)^2/[2]_\nu, \quad &&  K_n \mapsto -\nu\inv \bfk_n^{-2},
	\end{alignat*}
\emph{($ A_{2n }^{(2)}, \nu = q $)}	\hspace{-3cm}\centerline{ $ \pi_{A_{2n }^{(2)}, z } : U_q(A_{2n }^{(2)} ) \rightarrow \scrB_{\nu}^{\otimes n}[z, z\inv] $}	
\begin{alignat*}{4}
	&X_0^+ \mapsto z(\bfazh_1)^2/[2]_\nu, \quad \quad&& X_0^- \mapsto  z\inv\bfa_1^2/[2]_\nu, \quad\quad &&  K_0 \mapsto -\nu \bfk_1^2,\\
	&X_i^+ \mapsto \bfa_i\bfazh_{i+1}, \quad\quad && X_i^- \mapsto \bfazh_i\bfa_{i+1}, \quad\quad && K_i \mapsto \bfk_i\inv\bfk_{i+1}, \\
	& X_n^+ \mapsto \imath \tau_\nu\bfa_n, \quad \quad&& X_n^- \mapsto  \bfazh_n, \quad\quad &&  K_n \mapsto \imath \nu^{-\frac{1}{2}}\bfk_n^{-1},
\end{alignat*}
\emph{($ D_{n+1 }^{(2)}, \nu = q^2 $)}	\hspace{-3cm}\centerline{ $ \pi_{D_{n+1 }^{(2)}, z} : U_q(D_{n+1 }^{(2)} ) \rightarrow \scrB_{\nu}^{\otimes n}[z, z\inv] $}	
\begin{alignat*}{4}
	&X_0^+ \mapsto z\bfazh_1, \quad \quad\quad&& X_0^- \mapsto \imath\tau_\nu z\inv \bfa_1, \quad\quad \quad&&  K_0 \mapsto -\imath \nu^{\frac{1}{2}} \bfk_1,\\
	&X_i^+ \mapsto \bfa_i\bfazh_{i+1}, \quad\quad\quad && X_i^- \mapsto \bfazh_i\bfa_{i+1}, \quad\quad \quad&& K_i \mapsto \bfk_i\inv\bfk_{i+1}, \\
	& X_n^+ \mapsto \imath \tau_\nu\bfa_n, \quad\quad \quad&& X_n^- \mapsto  \bfazh_n, \quad \quad\quad&&  K_n \mapsto \imath \nu^{-\frac{1}{2}}\bfk_n^{-1},
\end{alignat*}	
where $\tau_\nu = {(\nu+1)}/{(\nu-1)}$. \qed
\end{proposition}

     \section{Multiplicity-free weight modules}\label{Sec:construction}
 In this section, we construct the multiplicity-free weight representations over $\Uqg$ from the solutions and the algebra homomorphisms in the previous section. Throughout this section, we assume that $\Uqg$ is the quantum affine algebra of type  $X_N^{(r)}$ in Proposition \ref{Prop:alghom}.

 \subsection{Module structures on $\calA_0$}\label{Subsection:U0free} In order to construct the multiplicity-free weight representations, we first consider the auxiliary $\Uqg$-module structures on $\calA_0=\bbC[z_1^{\pm1}, z_2^{\pm1}, \cdots, z_n^{\pm1}]$. 
 
 Let us fix some notations here.  Note that $\alpha_0$ and $\alpha_n$ are long roots in the type $C_{n}^{(1)}$, while both of them are short roots in $ D_{n+1}^{(2)} $. In addition, by our assumption, $\alpha_0$ is long, $\alpha_n$ is short in $A_{2n}^{(2)}$.  We define a pair $\kappa:=(\kappa_1,\kappa_2)$ of signs such that $\kappa_1$, $\kappa_2$ are equal to $0$ or $1$, which depends on the length of the roots $\alpha_0$ and $\alpha_n$ for each type, that is, 
  \[ (\kappa_1,\kappa_2)=(1,1) \;\text{for}\;   C_n^{(1)} ,  \quad   (\kappa_1,\kappa_2)=(1,0)  \;\text{for}\;    A_{2n}^{(2)}, \quad  (\kappa_1,\kappa_2)=(0,0)  \;\text{for}\;   D_{n+1}^{(2)}.  \]

 Fix a solution $(\phi_i)_{i\in I}$ of (\ref{mainequ}), and recall the units $z_i$'s  for each type, and the shift operators $\zeta_i$ defined in Subsection \ref{Subsec:ShiftCondition}. Put $b =b_{X_{N}^{(r)}} $. 
 Then we have 
 \begin{theorem}\label{Thm:Constructions}
 	Let $z$ be a parameter valued in $\bbC^\times$. For  an $n$-tuple $\bff = (f_i)_{1\leq i\leq n}$ satisfying that $f_i$ is  $1$ or $bz_i-b\inv z_i\inv$, ${1\leq i\leq n}$, there exists a  $\Uqg$-module structure on the algebra $\calA_0$ for each type defined in the following:

For the type $A_{n-1}^{(1)}$,
\begin{align*}
		X_i^+.u = z^{\delta_{i,0}}f_i\zeta_i(\frac{\{bz_{i+1}\}}{f_{i+1}})\zeta_i(u),  \quad  X_i^-.u =  z^{-\delta_{i,0}}\zeta_i\inv(\frac{\{bz_{i}\}}{f_{i}})f_{i+1}\zeta_i\inv(u),
\end{align*}
and $ K_i^{\pm 1}.u = y_i^{\pm 1}u $,  for any $u \in \calA_0$. 

For other types, 
\begin{equation*}
	\begin{aligned}
		& X_0^+.u =  z\frac{\zeta_0(\phi_0)\zeta_0(u)}{\zeta_1\inv(f_1)^{\kappa_1}\zeta_0(f_1)}, \quad && X_0^-.u = z\inv f_1 \zeta_1(f_1)^{\kappa_1}\zeta_0\inv(u), \\ 
		&X_i^+.u =  f_i\zeta_i(\frac{\{bz_{i+1}\}_i}{f_{i+1}})\zeta_i(u),  \quad  && X_i^-.u =  \zeta_i\inv(\frac{\{bz_{i}\}_i}{f_{i}})f_{i+1}\zeta_i\inv(u), \\
		&X_n^+.u =  f_n\zeta_{n-1}\inv(f_n)^{\kappa_2}\zeta_n(u),  \quad && X_n^-.u =  \frac{\phi_n\zeta_n\inv(u)}{\zeta_n\inv(f_{n})\zeta_{n-1}(f_n)^{\kappa_2}},
	\end{aligned}
\end{equation*}
and $ K_i^{\pm 1}.u = y_i^{\pm 1}u $,  for any $u \in \calA_0$. 
 \end{theorem}
\begin{proof}
Taking $u=1$ in the above construction we have the precise expressions of the actions $X_i^\pm.1$. In addition, for any $u\in \calA_0$ we have $X_i^\pm.u = \zeta_i^{\pm 1}(u)X_i^\pm.1$. We can check each defining relation directly. For the relations (\ref{Relations:commutation}), we have 
	\begin{align*}
	(K_iX^\pm_jK_i\inv).u = y_i\Big(X_j^\pm.(y_i\inv u)\Big) = y_i\zeta_j^{\pm 1}(y_i\inv u) X_j^\pm.1 
	= y_i\zeta_j^{\pm1}(y_i\inv) X_j^\pm.u 
	= q_i^{\pm a_{ij}}X_j^\pm.u.
\end{align*}
For the relations (\ref{Relations:XX}), we can split it into three cases:

(1) If $i=j$, then we have 
\begin{align*}
	(X_i^+X_i^--X_i^-X_i^+).u = u\Big(\zeta_i(X_i^-.1)X_i^+.1 - \zeta_i\inv(X_i^+.1)X_i^-.1\Big) = u\Big(\zeta_i(\phi_i) - \phi_i\Big).
\end{align*}
Here $\phi_i:= \zeta_i\inv(X_i^+.1)X_i^-.1$, $i \in I$ is just a solution to the system (4.1) by the construction, which implies (2.3) for $i=j$ as desired. 

(2) If $|i-j| > 1$, then we have $\zeta_i(f_k)=f_k$ and $\zeta_i(\{bz_k \}_k) = \{bz_k \}_k$ for $k = j , j+1$. Similarly, $\zeta_j(f_k)=f_k$ and $\zeta_j(\{bz_k \}_k) = \{bz_k \}_k$ for $k = i , i+1$. Therefore, we have 
\[X_i^+X_j^-.u = \zeta_i\zeta_j\inv(u)\zeta_i(X_j^-.1)X_i^+.1 = \zeta_i\zeta_j\inv(u)(X_j^-.1)(X_i^+.1) =X_j^-X_i^+.u.  \]

(3) The case that $|i-j| =1$. We need to do more detailed calculations for each type. First  assume $a_{ij}a_{ji} =1$. Then we have to show 
\[\zeta_i\zeta_j\inv\Big(\frac{\{bz_j\}_j}{f_j}\Big)\zeta_i(f_{j+1}) f_i \zeta_i\Big(\frac{\{bz_{i+1} \}_i}{f_{i+1}}\Big) = \zeta_j\inv\Big(\frac{\{bz_j \}_j}{f_j}\Big)f_{j+1}\zeta_j\inv(f_i)\zeta_i\zeta_j\inv\Big(\frac{\{bz_{i+1}\}_i}{f_{i+1}}\Big). \]
If $j=i+1$ then 
$ \zeta_i(f_{j+1})=f_{j+1} $, $ \zeta_j\inv(f_i)=f_i $,  and $ \zeta_i(f_j)=\zeta_j\inv(f_j) $; 
while $i=j+1$ we have $\zeta_j\inv(f_{i+1})=f_{i+1}$, $\zeta_i(f_j)=f_j$, and $\zeta_i(f_i) = \zeta_j\inv(f_i)$. Both cases imply the above equality holds. When $X_N^{(r)}\neq A_{n-1}^{(1)}$, a direct computation yields the following equalities: 
\begin{equation*}
	\begin{aligned}
		&\{bz_1 \}_1\zeta_1\phi_0 = \phi_0\zeta_0\inv\{bz_1\}_1, \quad && f_1^{\kappa_1}\zeta_0^{\pm1}\zeta_1^{\pm1}(f_1) = f_1\zeta_1^{\mp1}(f_1)^{\kappa_1}, \\ 
		&\{bz_n \}_{n-1}\phi_n = \zeta_n^\inv(\{bz_n\}_{n-1})\zeta_{n-1}^\inv(\phi_n), \quad && f_n^{\kappa_2}\zeta_{n-1}^{\pm1}\zeta_{n}^{\pm1}(f_n) = f_n\zeta_{n-1}^{\mp1}(f_n)^{\kappa_2}. 
	\end{aligned}
\end{equation*}
Then similar arguments for the case that $a_{ij}a_{ji} = 2$ are true. Any tuple $(\phi_i)_i$ satisfying (\ref{mainequ}) and the choice of $(f_i)_i$ also guarantee that these actions hold under the quantum Serre relations (\ref{Relations:XX+}). 
\end{proof}

 
 Denote the above $\Uqg$-module on $\calA_0$ related to $z$ and $\bff$ by $S_z(\bff)$.  Recall $\pi_{X_N^{(r)}}(K_i)= y_i$ in the construction in Subsection \ref{Subsec:quantumoscillator}. Then we have 
 \begin{align}\label{Equa:centeraction}
 \pi_{X_N^{(r)}}(K_\delta)= \prod_{i \in I}y_i^{a_i}=\prod_{i \in I}\prod_{j\in I}x_j^{a_{ji}a_i}=\prod_{j \in I}x_j^{\sum_{i\in I}a_{ji}a_i}
 =1.
 \end{align}
  In particular, $K_\delta$ acts trivially on $S_z(\bff)$. Therefore $S_z(\bff)$ is finitely $U^0$-generated instead of $U^0$-diagonalizable when restricted as a $U^0$-module.
  
 Now let us explain the ``weighting'' procedure mentioned in the introduction. That is,  to a  $\Uqg$-module $S_z(\bff)$, we associate a weight module $\calW(S_z(\bff))$ in the following way. Consider the algebra homomorphism from $U^0$ to $\calA_0$ by assigning $K_i$ to $y_i$, $i\in I_0$, which induces a natural group homomorphism from the group of characters of $\calA_0$ to $\scrX$. For any character $\varphi $ of $\calA_0$, denote by $\frakm_\varphi (:=\ker \varphi)$ the corresponding maximal ideal of $\calA_0$. Extend $\alpha_j\in \scrX$ to a character of $\calA_0$ by setting $\alpha_j(y_i) =q_j^{a_{ji}}$, we still denote it by $\alpha_j$,
 then we have  
  \[ K_i.\frakm_\varphi S_z(\bff) \subset \frakm_\varphi S_z(\bff), \quad X_i^{\pm }.\frakm_\varphi S_z(\bff) \subset \frakm_{\varphi\pm \alpha_i}S_z(\bff).  \]
  
  Define
  \[\calW(S_z(\bff)) = \bigoplus_{\varphi}S_z(\bff)/\frakm_\varphi S_z(\bff), \]
  where $\varphi$ is taken over all characters of $\calA_0$. 
  \begin{corollary}
  	For any $\Uqg$-module $S_z(\bff)$, we have $\calW(S_z(\bff))$ is a weight module, and all its simple subquotients are multiplicity-free. 
  \end{corollary}
\begin{proof}
	It is clear that $S_z(\bff)/\frakm_\varphi S_z(\bff)$ is $1$-dimensional and $K_i$ acts  diagonally. In particular, $K_\delta$ acts by $1$. The first assertion follows from the previous statements, and the $\lambda$-weight space
	$ \calW(S_z(\bff))_{\lambda} = \oplus_{\bar{\varphi}=\lambda}S_z(\bff)/\frakm_\varphi S_z(\bff)$, where $\bar{\varphi}$ means the image of $\varphi$ in $\scrX$. Since we have 
	\[ X_i^{\pm }.S_z(\bff)/\frakm_\varphi S_z(\bff) \subset S_z(\bff)/\frakm_{\varphi\pm \alpha_i}S_z(\bff)\]
	the second assertion follows. 
\end{proof}
  
   \begin{remark}
   	 In fact, the $\Uqg$-module $ \calW(S_z(\bff)) $ is a $q$-analog of the \emph{coherent family} in the sense of  \cite{Mat}. This ``weighting'' procedure was first suggested by O. Mathieu in the paper \cite{Nil16}. 
    \end{remark}
   
   Now let us study the possible highest weights of $\calW(S_z(\bff))$ when restricted as a $U_q(\fing)$-module.  
   Assume that the weight vector $1+\frakm_\varphi S_z(\bff)$ of $ \calW(S_z(\bff)) $ is a highest weight vector for some $\varphi$. Then we have $X_i^+.(1+\frakm_\varphi S_z(\bff))=0$ for $i\in I_0$, which implies that
   \begin{equation} \label{solutions} (\varphi+\alpha_i)(\zeta_i(\phi_i)) = 0 \quad \text{ for }\quad  i\in I_0. 
   	\end{equation}
   The weight $\lambda = \bar{\varphi}$ is level-zero, which is determined uniquely by the values $\lambda(K_i), i \in I_0$. Therefore, all level-zero weights can be seen automatically as weights over $U_q(\fing)$.  As a result, we can obtain the following result. 
   \begin{corollary}\label{Highestweight}
   	Let $\lambda \in \scrX$ be a  weight of   $\calW(S_z(\bff))$ for some $\bff$. If $\lambda$ is a highest $U_q(\fing)$-weight, then up to twistings by the automorphisms of $U_q(\fing)$, we have 
   	\begin{enumerate}[\rm (1)]
   		\item for the type $A_{n-1}^{(1)}$, the weight $\lambda$  is of the  form $\omega_0 +a \omega_s- (a+1) \omega_{s+1}$ for some $  a \in \bbC  $ and $s \in I$ up to a constant multiple of $\delta$;

   		\item for the type $C_{n}^{(1)}$, the weight $\lambda$ has the form 
   	$ 	\omega_0/2 + \omega_{n-1}-3\omega_n/2 $ or  $  (\omega_0 - \omega_{n})/2 $ up to a constant multiple of $\delta$;
   	
      	\item for the type $A_{2n}^{(2)}$ (resp.  $D_{n+1}^{(2)}$), the weight  $\lambda = (\lambda(K_0), \cdots, \lambda(K_n))$ is defined as
      	\[(-q, 1, \cdots, 1, \imath q^{-1/2}) \quad (\text{resp.} \quad (-\imath q, 1, \cdots, 1, \imath q\inv)).  \]
   	\end{enumerate}
   \end{corollary}
   
\begin{proof}
   The result can be deduced directly from (\ref{solutions}). For example, in the type $ A_{n-1}^{(1)} $, these equations   (\ref{solutions}) become
   	 \begin{equation}\label{TypeA}
   	\begin{dcases}
   	m_0m_1\cdots m_{n-1} = 1, \\
   		\{qbm_i \}\{bm_{i+1} \} = 0, \quad 1 \leq i \leq n-1,
   	\end{dcases}
   \end{equation}
where we denote $ \varphi(z_i)$ by $ m_i  $ for $i\in I$. Let $\lambda = \bar{\varphi}$. Then $\lambda(K_i) = m_im_{i+1}\inv$. To solve the equations (\ref{TypeA}), we divide it into two cases: if $\{qbm_1 \}$ is not zero, then $ \{bm_{i} \} =0  $ for $i \geq 2$; if $\{qbm_1 \}$ is zero, then we assume that $s$ is the maximal index such that $\{qbm_s \}$ is zero, then $\{qbm_i \} =0$ for $i\leq s$ and $ \{bm_{j} \} =0 $ for $j\geq s+2$. In the first case, the possible solutions are $m_1=\pm b^{n-1}$, $m_i = \pm b\inv$ for $i\neq 1$. Then up to twistings by sign automorphisms of $U_q(\fing)$, we have the weight $\lambda$ is given by 
\[\lambda(K_0)=b^{-n}, \quad \lambda(K_1)=b^{n},\quad \lambda(K_i)=1, \quad i \geq 2, \]
which is of the form $(a+1)\omega_0-(a+1)\omega_1$ for some $a\in \bbC$. In the second case, $m_i = \pm q\inv b\inv$ for $0\leq i \leq s$, $m_{s+1}=\pm q^{s+1}b^{n-1}$, and $m_j=\pm b\inv$ otherwise. Then up to signs, we have the weight $\lambda$ is the following:
\[\lambda(K_s)=q^{-s-2}b^{-n}, \quad \lambda(K_{s+1})=q^{s+1}b^n, \quad \lambda(K_i) =1, \quad i\neq s, s+1,  \]
which is exactly of the form in (1). So the assertion (1) follows. 
\end{proof}
   
   All simple subquotients obtained in  Corollary  \ref{Highestweight} can be realized as $q$-oscillator representations by using the Fock space representations of $\scrB_{\nu}$ and the algebra homomorphisms in Proposition \ref{Prop:alghom} (cf. e.g. \cite{K18}). In the following subsection, we shall recall the $q$-oscillator representations.

\subsection{Realization of multiplicity-free weight modules}\label{Subsection:Fockspacerep} Let $F = \oplus_{m \in \bbZ_{\geq 0}}\bbC\cket{m}$ be the Fock space representation of $\scrB_\nu$ on which the generators $\bfazh$ and $\bfa$ act as the creation and annihilation operators respectively, and the element $\bfazh \bfa = \{\bfk \}_\nu$ corresponds to the number operator, more precisely, for any $m \in \bbZ_{\geq 0}$, 
\[\bfazh.\cket{m} = \cket{m+1}, \quad \bfa.\cket{m} = [m]_\nu \cket{m-1}, \quad \bfk^{\pm1}.\cket{m} = \nu^{\pm m} \cket{m}. \]
In particular, $\bfa.\cket{0}=0$. 

Denote this representation by $\rho: \scrB_{\nu} \rightarrow \End_{\bbC}(F)$. For any $b \in \bbC^\times$ and $\varepsilon \in \{0,1\}$, we denote by $\rho_{\varepsilon, b} $ the composition $ \rho\circ \vartheta^{\varepsilon}\circ \theta_{b,0}$ (cf. Lemma \ref{Lemma:automorphisms}). Then $F$ has a new $\scrB_\nu$-module structure via $\rho_{\varepsilon, b} $.

\begin{definition}\label{Def:q-oscillator}
	Let $z$ be a parameter valued in $\bbC^\times$. We define the representation  $\calF^z_{\bfepsilon,\bfb}$ of $\Uqg$ on the space $ F^{\otimes n} $ via the composition of the algebra homomorphisms $\pi_z:=\pi_{X_N^{(r)}, z}$ defined in Proposition $ \ref{Prop:alghom} $ and 
\[ \scrB_{\nu}^{\otimes n}[z, z\inv] \xrightarrow{\rho_{\varepsilon_1, b_1}\otimes\cdots\otimes\rho_{\varepsilon_n, b_n}} \End_{\bbC}(F^{\otimes n}) \]  
where $\bfepsilon = (\varepsilon_i)_i \in \{0, 1 \}^n$, and $\bfb = (b_i)_i \in (\bbC^\times)^n$.
\end{definition}

 For any $n$-tuple $(m_i)_i \in (\bbZ_{\geq 0})^n$, we use $\cket{\bfm}:=\cket{m_1}\otimes\cdots\otimes \cket{m_n}$ for the basis vector of $\calF^z_{\bfepsilon,\bfb}$. Let $\bfe_j$ be the $j$-th standard vector in $\bbZ^n$ with $1$ at the $j$-th term and $0$ otherwise for $1\leq j \leq n$. Moreover, set $\bf 0$ for $(0, \cdots, 0) \in \bbZ^n$.

For $n \geq 2$, note that $\Uqg$-module actions of $X^{\pm1}_i, K_i$ for $1\leq i \leq n-1$ on $\calF^z_{\bfepsilon,\bfb}$, by Definition \ref{Def:q-oscillator}, can be written down uniformly as follows: 
 \begin{align}
 	\label{Equation:X^+}&X_i^+.\cket{\bfm}  = (-1)^{\varepsilon_{i}}b_ib_{i+1}\inv [{m_i}/{m_i^{\varepsilon_i}}]_\nu [m_{i+1}^{\varepsilon_{i+1}}]_\nu \cket{\bfm-(-1)^{\varepsilon_i}\bfe_i+(-1)^{\varepsilon_{i+1}}\bfe_{i+1}}, \\
 	&X_i^-.\cket{\bfm} = (-1)^{\varepsilon_{i+1}}b_i\inv b_{i+1}[m_{i}^{\varepsilon_{i}}]_\nu [{m_{i+1}}/{m_{i+1}^{\varepsilon_{i+1}}}]_\nu  \cket{\bfm+(-1)^{\varepsilon_i}\bfe_i-(-1)^{\varepsilon_{i+1}}\bfe_{i+1}}, \\
 	\label{Equation:weightaction}&K_i.\cket{\bfm}  =\nu^{-(-1)^{\varepsilon_i}(m_i+\varepsilon_i)+(-1)^{\varepsilon_{i+1}}(m_{i+1}+\varepsilon_{i+1})}\cket{\bfm},
 \end{align}
for  $\bfm \in (\bbZ_{\geq 0})^n$, where $\nu$ is defined in Proposition \ref{Prop:alghom} for each type, and $\cket{\bfm}$ for $\bfm \notin (\bbZ_{\geq 0})^n$ can be read as $0$. Here we remark that the $\UqA$-module actions of $X^{\pm}_0, K_0$ on $\cket{\bfm}$ also have the above forms where we understand the indices $i, i+1$ as $n, 1$ (mod $n$) respectively. 

Regard $U_q(A_{n-1})$ as the subalgebra of $U_q(\frakg) (n \geq 2)$ via forgetting the actions of the Drinfeld-Jimbo generators indexed by $0$ and $n$.  One can check that  $\calF^z_{\bfepsilon,\bfb}$ as a $U_q(A_{n-1})$-module  is a multiplicity-free weight module. In fact, $\calF^z_{\bfepsilon,\bfb}$ has the following direct sum decomposition:
\[\calF^z_{\bfepsilon,\bfb} = \bigoplus_{l = -\infty}^\infty \calF_{\bfepsilon,\bfb}^{z, (l)}, \quad \calF_{\bfepsilon,\bfb}^{z, (l)}= \bigoplus_{|\bfm|_\bfepsilon = l}\bbC\cket{\bfm}. \]
For any $\bfm \in \bbZ^n$, denote $|\bfm|_\bfepsilon = \sum_i(-1)^{\varepsilon_i}m_i$. Each $ \calF_{\bfepsilon,\bfb}^{z, (l)} $ is an irreducible, multiplicity-free weight $U_q(A_{n-1})$-module by the formulae (\ref{Equation:X^+})-(\ref{Equation:weightaction}) as $q$ (or $\nu$) is not a root of unity. 

Fix $1 \leq i \leq n$. Define the algebra homomorphism $\delta_i: \calA_0 \rightarrow \bbC$ by $ \bfk_j \mapsto \nu^{\delta_{ij}} $ for $1 \leq j \leq n$. Then it induces an algebra character  $\tdelta_{i}\in \scrX$ by
\[\tdelta_{i} : U^0 \stackrel{\pi|_{U^0}}{\longrightarrow} \calA_0 \stackrel{\delta_i}{\longrightarrow} \bbC. \]
Then we have

\begin{proposition}
For $\bfepsilon \in \{0,1\}^n$, $\bfb \in (\bbC^\times)^n$, we have 
\begin{enumerate}[\rm (i)]
	 \item $\calF^z_{\bfepsilon,\bfb}$ is a weight module with  $\dim (\calF^z_{\bfepsilon,\bfb})_\lambda \leq 1$ for any $\lambda \in \scrX$. 
	 \item If $\dim (\calF^z_{\bfepsilon,\bfb})_\lambda = 1$, then there exists $\bfm \in (\bbZ_{\geq 0})^n$ such that $ (\calF^z_{\bfepsilon,\bfb})_\lambda = \bbC\cket{\bfm}$ with
	\begin{equation}\label{Equation:weight}
		\lambda=  \sum_{i=1}^n(-1)^{\varepsilon_i}(m_i+\varepsilon_i)\tdelta_i. 
	\end{equation}
\end{enumerate}	
\end{proposition}
\begin{proof}
	It is clear that $\calF^z_{\bfepsilon,\bfb}$ is a weight module. By (\ref{Equation:weightaction}) and the actions of $K_0$ and $K_n$, which are defined for $X_N^{(r)} \neq A_{n-1}^{(1)}$ as follows: 
	\begin{align}
		K_0.\cket{\bfm}&= -\imath^{1-\kappa_1}\nu^{(-1)^{\varepsilon_1}(\kappa_1+1)(m_1+1/2)}\cket{\bfm}, \\
		K_n.\cket{\bfm}&= \imath^{1+\kappa_2}\nu^{-(-1)^{\varepsilon_n}(\kappa_2+1)(m_n+1/2)}\cket{\bfm},
	\end{align}
	where $\kappa_1$ and $\kappa_2$ are defined in Subsection \ref{Subsection:U0free}, the relative weight of $\cket{\bfm}$ is given by the right hand side of the equality (\ref{Equation:weight}). By the above statement, $\dim(\calF^z_{\bfepsilon,\bfb})_{\lambda}\leq 1$ for any $\lambda \in \scrX$. 
\end{proof}

Consider the following decomposition of $\calF^z_{\bfepsilon,\bfb}$:
\[\calF^z_{\bfepsilon,\bfb} = \calF_{\bfepsilon,\bfb}^{z,+} \oplus \calF_{\bfepsilon,\bfb}^{z,-}, \quad \calF_{\bfepsilon,\bfb}^{z,+} =\bigoplus_{|\bfm|_\bfepsilon \equiv 0 (\Mod 2)}\bbC\cket{\bfm}, \quad  \calF_{\bfepsilon,\bfb}^{z,-}=\bigoplus_{|\bfm|_\bfepsilon \equiv 1 (\Mod 2)}\bbC\cket{\bfm}. \]
For $0 \leq s \leq  n$, let $\bfepsilon_{>s}\in \{0,1\}^n$ satisfy
\[\varepsilon_1=\cdots=\varepsilon_{s}=0, \quad \varepsilon_{s+1} = \cdots =\varepsilon_n=1. \]
For example, $\bfepsilon_{>0}= (1, \cdots, 1)$ and $\bfepsilon_{>n}= (0, \cdots, 0)$. 

Then we have 

\begin{proposition}\label{Prop:Highestweights}
  For any $\bfepsilon \in \{0,1\}^n$, $\bfb \in (\bbC^\times)^n$, we have 
  \begin{enumerate}[\rm (i)]
  	\item As a $U_q(A_{n-1}^{(1)})$-module, $\calF_{\bfepsilon,\bfb}^{z,(l)}$ is irreducible for any admissible $l \in \bbZ$ (defined in (\ref{Equ:v_{ls}})); it is a highest $\ell$-weight module with a highest $\ell$-weight vector $v_{l,s}$  if and only if $\bfepsilon = \bfepsilon_{>s}$ for some $0\leq s \leq n$, where
  	  \begin{equation}\label{Equ:v_{ls}}
  	 	v_{l,s}=\left\{
  	 	\begin{array}{l}
  	 		\cket{l\bfe_s}, \quad l \geq 0 \text{ and } 0 < s \leq n, \\
  	 		\cket{-l\bfe_{s+1}}, \quad l <0 \text{ and } 0 \leq s <n.
  	 	\end{array}
  	 	\right.
  	 \end{equation} 
  	\item As $U_q(C_{n}^{(1)})$-modules, $\calF_{\bfepsilon,\bfb}^{z,+}$ and $ \calF_{\bfepsilon,\bfb}^{z,-} $ are irreducible; they are highest $\ell$-weight modules with highest $\ell$-weight vector $v^+ = \cket{\bf 0}$ and $v^- = \cket{\bfe_n}$ respectively whenever $\bfepsilon = \bfepsilon_{>n}$.  
  	\item As a $U_q(A_{2n}^{(2)})$ or $U_q(D_{n+1}^{(2)})$-module,  $\calF^z_{\bfepsilon,\bfb}$ is irreducible; it is a highest $\ell$-weight module with a highest $\ell$-weight vector $v = \cket{\bf 0}$ whenever $\bfepsilon = \bfepsilon_{>n}$.  
  \end{enumerate}
\end{proposition}
\begin{proof}
	 Note that $K_\delta$ acts trivially on $\calF^z_{\bfepsilon,\bfb}$  by (\ref{Equa:centeraction}). The defining relations (\ref{DrinfeldRelations}) imply that the actions of $h_{i,k}$, $1\leq i \leq N$, $k\in \bbZ\setminus\{0\}$ on $\calF^z_{\bfepsilon,\bfb}$ commute pairwise. Hence $\calF^z_{\bfepsilon,\bfb}$ is an $\ell$-weight $\Uqg$-module. It is clear that $\calF_{\bfepsilon,\bfb}^{z,(l)}$ is closed under the action of $\UqA$. The irreducibility of $\calF_{\bfepsilon,\bfb}^{z,(l)}$ can be checked by the actions (\ref{Equation:X^+})-(\ref{Equation:weightaction}). Note that for $\bfepsilon = \bfepsilon_{>0}$ (resp. $\bfepsilon = \bfepsilon_{>n}$), $\calF_{\bfepsilon,\bfb}^{z,(l)}$ is finite dimensional for $l < 0$ (resp. $l \geq 0$). As a $U_q(A_{n-1})$-module (ignore the actions of $X_0^\pm, K_0$) $\calF_{\bfepsilon,\bfb}^{z,(l)}$ is a highest weight module iff $\varepsilon_1 \leq \cdots \leq \varepsilon_n$, and the corresponding highest weight vector can be chosen as (\ref{Equ:v_{ls}}), which is also a highest $\ell$-weight vector by weight consideration. 
	
	For $X_N^{(r)} \neq A_{n-1}^{(1)}$, the actions of $X_0^+$ and $X_n^+$ are given by 
	\begin{equation*}
		\hspace{-.7cm} X_0^+.\cket{\bfm} =z \frac{b_1^{-\kappa_1-1}}{[\kappa_1+1]_\nu}\prod_{j=0}^{\kappa_1}[(m_1-j)^{\varepsilon_1}]_\nu \cket{\bfm + (-1)^{\varepsilon_1}(\kappa_1+1)\bfe_1},
	\end{equation*}
and $ X_n^+.\cket{\bfm} =  x\cket{\bfm - (-1)^{\varepsilon_n}(\kappa_2+1)\bfe_n} $, where $a\in \bbC^\times$ is defined as follows: 
   	\begin{align*}
   	 x=(-1)^{\varepsilon_n(1-\kappa_2)} b_n^{\kappa_2+1} \frac{\imath^{1-\kappa_2}}{[\kappa_2+1]_\nu}\tau_{\nu,\kappa_2}\prod_{j=0}^{\kappa_2}[(m_n-j)/(m_n-j)^{\varepsilon_n}]_\nu
   \end{align*}
and $\tau_{\nu,\kappa_2}=(\nu-\kappa_2+1)/(\nu+\kappa_2-1) $. Similarly, we can obtain the actions of $X_0^-$ and $X_n^-$. Therefore, the assertions (ii) and (iii) can be deduced directly from the above actions. 
\end{proof}

\section{Highest $\ell$-weights}\label{Sec: highestlweights}
In this section, we focus on the irreducible highest $\ell$-weight representations constructed in the previous section, and compute their highest $\ell$-weights explicitly. 

\subsection{Multiplicity-free highest $\ell$-weight modules}\label{Subsection:q-oscillator representations} 

Fix $0 < s < n$. Let $\calW_s := F^{\otimes n }$ be the $U_q(A_{n-1}^{(1)})$-module  defined as follows (see also \cite{KL2022}): 
\begin{alignat*}{4}
	&   X_0^+.\cket{\bfm}  =   \cket{\bfm+\bfe_1+\bfe_n}, \quad &&
		X_0^-.\cket{\bfm} = -[m_1][m_n] \cket{\bfm-\bfe_1-\bfe_n},  \\
	&	X_s^+.\cket{\bfm}  =- [m_s][m_{s+1}] \cket{\bfm-\bfe_s-\bfe_{s+1}}, \quad &&
		X_s^-.\cket{\bfm} =   \cket{\bfm+\bfe_s+\bfe_{s+1}},\\
	& X_i^+.\cket{\bfm} =[m_i]\cket{\bfm -\bfe_i+ \bfe_{i+1}},\quad && X_i^-.\cket{\bfm} = [m_{i+1}]\cket{\bfm + \bfe_i-\bfe_{i+1}},  \\
	& X_j^+.\cket{\bfm} = [m_{j+1}]\cket{\bfm +\bfe_j-\bfe_{j+1}}, \quad && X_j^-.\cket{\bfm} =[m_j]\cket{\bfm -\bfe_j+\bfe_{j+1}}, 
\end{alignat*}
and 
\begin{alignat*}{3}
&K_0.\cket{\bfm} = q^{m_1+m_n+1}\cket{\bfm}, \quad && K_s.\cket{\bfm}  =q^{-m_s-m_{s+1}-1}\cket{\bfm},\\  
& K_i.\cket{\bfm} = q^{m_{i+1}-m_i}\cket{\bfm}, \quad && K_j.\cket{\bfm} = q^{m_j-m_{j+1}}\cket{\bfm}, 
\end{alignat*}
where $ 1\leq  i < s < j \leq n-1$ and   $\bfm \in (\bbZ_{\geq 0})^n$. From the actions (\ref{Equation:X^+})-(\ref{Equation:weightaction}), $\calW_s$ is just the twisting of the module $\calF^1_{\bfepsilon, \bfb}$ where $\bfepsilon = \bfepsilon_{>s}$, $\bfb = (1,\cdots, 1)$ by the automorphism of $U_q(A_{n-1}^{(1)})$ sending $X_k^\pm $ to $-X_k^\pm$, $s\leq k\leq n$ along with other Drinfeld-Jimbo generators fixed. Denote $ \calW_s^{(l)} $ as the $l$-th irreducible component of $\calW_s$, i.e.,  $\calW_s^{(l)} = \oplus_{|\bfm|_\bfepsilon = l}\bbC \cket{\bfm}$.

  Let  $(X_N^{(r)}, \nu)$ be one of the types in Proposition \ref{Prop:alghom} except $(A_{n-1}^{(1)}, q)$.  Let $\calW= F^{\otimes n }$ be the $U_q(\AffX)$-module defined as (see  \cite{KO15}):
\begin{align*}
	&   X_0^+.\cket{\bfm}  = \frac{1}{[\kappa_1+1]_\nu}  \cket{\bfm+(\kappa_1+1)\bfe_1}, \\
	&	X_0^-.\cket{\bfm} = -(- 1)^{|\kappa|} \frac{\imath^{1-\kappa_1}}{[\kappa_1+1]_\nu} \tau_{\nu,\kappa_1} \prod_{j=0}^{\kappa_1}[m_1-j]_\nu \cket{\bfm-(\kappa_1+1)\bfe_1}, \\
	&	K_0.\cket{\bfm} = (-1)^{|\kappa|}\imath^{1-\kappa_1}\nu^{(\kappa_1+1)(m_1+1/2)}\cket{\bfm},\\
	&	X_i^+.\cket{\bfm}  =  [{m_i}]_\nu  \cket{\bfm-\bfe_i+\bfe_{i+1}}, \\
	&	X_i^-.\cket{\bfm} =  [{m_{i+1}}]_\nu  \cket{\bfm+\bfe_i-\bfe_{i+1}}, \quad \quad (1\leq i \leq n-1)\\
	&	K_i.\cket{\bfm}  =\nu^{-m_i+m_{i+1}}\cket{\bfm},\\
	& X_n^+.\cket{\bfm} =  \frac{ \imath^{1+\kappa_2}}{[\kappa_2+1]_\nu}\tau_{\nu,\kappa_2}\prod_{j=0}^{\kappa_2}[m_n-j]_\nu\cket{\bfm - (\kappa_2+1)\bfe_n},\\
	& X_n^-.\cket{\bfm} = \frac{1}{[\kappa_2+1]_\nu}\cket{\bfm + (\kappa_2+1)\bfe_n},\\
	&K_n.\cket{\bfm} = \imath^{1-\kappa_2}\nu^{-(\kappa_2+1)(m_n+1/2)}\cket{\bfm},
\end{align*}
where  $\bfm \in (\bbZ_{\geq 0})^n$, $|\kappa| = \kappa_1+\kappa_2$ and $\tau_{\nu,\kappa_i}= (\nu-\kappa_i+1)/(\nu+\kappa_i-1)$. Here $\kappa_i$'s are defined in Subsection \ref{Subsection:U0free}. This module can be obtained from $\calF^1_{\bfepsilon, \bfb}$ with $\bfepsilon = \bfepsilon_{>n}$ and $\bfb = (1,\cdots, 1)$ by the automorphism of $U_q(\AffX)$ defined as
\[ X_0^- \mapsto (-1)^{|\kappa|+1}X_0^-, \quad K_0 \mapsto (-1)^{|\kappa|+1}K_0, \quad X_n^+ \mapsto (-1)^{\kappa_2}X_n^+, \quad K_n \mapsto (-1)^{\kappa_2}K_n \]
and other generators are fixed. For the type $\AffC$,  denote the irreducible components $\calF_{\bfepsilon, \bfb}^{1,\pm}$ of $U_q(\AffC)$-module $\calW$ by $\calW^\pm$, for convenience.

\begin{lemma}\label{Lemma:l-weight}
	Let $L(\bsf)$ be an irreducible highest $\ell$-weight $\Uqg$-module with $\bsf=(f_i(z))_{i\in I_0}\in \calR$. If $\dim L(\bsf)_{\wt(\bsf)-\alpha_i}=1$ for some $i \in I_0$ then $f_i(z)$ satisfies that
	\begin{equation}\label{Equforms}
	f_i(z) = f_{i,0}^+ \frac{1-(a-b)z}{1-az}
	\end{equation}
	where $a, b \in \bbC$ satisfy that $f_{i, 2\td_i}^+ = af_{i,\td_i}^+$ and $ f_{i, \td_i }^+= bf_{i, 0 }^+ $. 
\end{lemma}
\begin{proof}
	Suppose that  $v\in L(\bsf)$ is a nonzero $\ell$-weight vector of  $\bsf$. Note that $\{x_{i,k}^-.v, k \in \bbZ \}$ spans the weight space $L(\bsf)_{\wt(\bsf)-\alpha_i}$. If $ \dim (L(\bsf)_{\wt(\bsf)-\alpha_i})=1 $, then  there exist $j\in \bbZ$ such that $ x_{i, \tilde{d}_ij}^-.v $ is nonzero, and $a\in \bbC$ such that 
	\begin{equation}\label{Equ:action}
		x_{i, \tilde{d}_i(j+1)}^-.v=ax_{i, \tilde{d}_ij}^-.v. 
	\end{equation}
	Consider the actions of $x_{i, \tilde{d}_ik}^+$, $k\in \bbZ$ on (\ref{Equ:action}). The defining relations (\ref{DrinfeldRelations}) imply that 
	\begin{equation*}
		f_{i, \tilde{d}_i(k+j+1)}^+-f_{i, \tilde{d}_i(k+j+1)}^-=a(f_{i, \tilde{d}_i(k+j)}^+-f_{i, \tilde{d}_i(k+j)}^-)
	\end{equation*}
	for any $k\in \bbZ$. Since $f_{i,k}^-=f_{i,-k}^+ =0$ for $k>0$, we have $f_{i, \td_i(k+1)}^+ = af_{i,\td_ik}^+$ for any $k>0$. Take the series $f_i(z) = \sum_{k=0}^{\infty}z^kf_{i, \td_i k}^+$, we have 
	\begin{align*}
		f_i(z)(1-az) &= \sum_{k=0}^{\infty}z^kf_{i, \td_i k}^+ - a \sum_{k=0}^{\infty}z^{k+1}f_{i, \td_i k}^+ \\
		&= f_{i,0}^+ + \sum_{k=1}^{\infty}z^k(f_{i, \td_i k}^+- af_{i, \td_i(k-1)}^+) \\
		&= f_{i,0}^+ +(f_{i, \td_i }^+-af_{i, 0 }^+)z.
	\end{align*} 
	Hence $ f_i(z)  $ has the rational form (\ref{Equforms}).
\end{proof}

\subsection{Highest $\ell$-weights} Let us first study some properties of the Weyl group and the description of the root vectors of quantum affine algebras, which will enable us to compute the highest $\ell$-weight explicitly. 

\begin{lemma}\label{Lemma:Weylgroupelements}
	Let $i,j \in I$, and $i \neq j$.
	\begin{enumerate}[\rm (1)]
		\item If $a_{ij}a_{ji}=1$, then $s_js_i\alpha_j = \alpha_i$. 
		\item If $a_{ij}a_{ji}=2$, then $s_is_js_i\alpha_j =\alpha_j$. 
	\end{enumerate}
\end{lemma}
\begin{proof}
		Both (1) and (2) are easy facts deduced from $ s_js_i\alpha_j = (a_{ij}a_{ji}-1)\alpha_j-a_{ij}\alpha_i  $ and 
	\begin{equation*}
	 s_is_js_i\alpha_j  = (a_{ij}a_{ji}-1)\alpha_j+(2-a_{ij}a_{ji})a_{ij}\alpha_i
	\end{equation*}
respectively. 
\end{proof}

Recall the braid group operators associated to $\tW$ introduced by Lusztig \cite{Lus90}. For each simple reflection $s_i$,  there is an algebra automorphism $T_i = T_{s_i}$ of $\Uqg$ defined by
\begin{alignat*}{4}
	&T_{i}X^+_{i} = -X^-_{i}K_{i}, \quad T_{i}X^-_{i}= - K_{i}\inv X^+_{i}, \quad T_{i}K_\beta=K_{s_i\beta},\\
	&T_{i}X^+_{j}=
	\sum_{k=0}^{-a_{ij}} (-1)^{k-a_{ij}}
	q_{i}^{-k} (X^+_{i})^{(-a_{ij}-k)}X^+_{j}(X^+_{i})^{(k)},  (i \neq j), \\
	& T_{i}X^-_{j}=
	\sum_{k=0}^{-a_{ij}} (-1)^{k-a_{ij}}
	q_{i}^{k} (X^-_{i})^{(k)}X^-_{j}(X^-_{i})^{(-a_{ij}-k)},  (i \neq j),
\end{alignat*}
where  $\beta\in Q$ and $ (X_i^{\pm})^{(k)} = (X_i^{\pm})^k/[k]^!_{i} $. 
Then $\Phi T_i = T_i\inv \Phi$, where $\Phi$ is the $\bbC$-linear anti-automorphism of $\Uqg$ sending $X_i^\pm$ to $X_i^\pm$, $K_i$ to $K_i\inv$ for $i\in I$. For any $\tau \in \scrT$, define $T_\tau$ by $T_\tau(X_i^\pm) = X_{\tau(i)}^\pm$ and $T_\tau(K_i) = K_{\tau(i)}$.

For later use, we list some well-known properties of braid group operators (cf. \cite{Lus93,Beck94}).  Choose one element $w \in \tW$.  If $\tau s_{i_1}s_{i_2}\cdots s_{i_m}$ is a reduced expression of $w$, then the automorphism $T_w = T_\tau T_{i_1}T_{i_2}\cdots T_{i_m}$ of $\Uqg$ is independent on the choice of the reduced expression of $w$. In particular, one reduced expression can be transformed to another by a finite sequence of braid relations. If $w(\alpha_i) = \alpha_j$ then $T_w(X_i^+) = \Xzh{j}$. Moreover, if $w = s_{i_1}s_{i_2}\cdots s_{i_m}$ is a reduced expression  and $l(ws_i)= l(w)+1$,  then we have  $T_{w}\Xzh{i}\in U^+$.

\begin{remark}\label{Remark:Reducedexpression}
	For $i \leq j$, put $ \bfs_{(i,j)} = s_is_{i+1}\cdots s_j. $ 

\noindent	$(\rm i)$	In the type $A_{n-1}^{(1)}$, a reduced expression of $\tomega_i$, $1 \leq i \leq n-1$ can be chosen as: 
		\[\tomega_i = \tau^i \bfs_{(1,n-i)}^{-1}\bfs_{(2,n-i+1)}^{-1}\cdots \bfs_{(i,n-1)}^{-1}  \]
		where $\tau$ is the diagram automorphism of $A_{n-1}^{(1)}$ sending $j$ to $j+1 (\mathrm{mod}\; n)$ for $j \in I$ (cf. \cite[Subsection 3.3]{JKP21}). 
		
			\noindent	$(\rm ii)$		In the type  $A_{2n}^{(2)}$, the reduced expression of $\tomega_n$ can be chosen (cf. \cite[Corollary 4.2.4]{Dami98}) as:
		\[\tomega_n = (s_0s_1\cdots s_n)^n. \]
		
\noindent	$(\rm iii)$		In the type $C_{n}^{(1)}$ (resp.  $D_{n+1}^{(2)}$), the reduced expressions of $\tomega_{n-1}$ and $\tomega_n$ can be chosen as:
		\[\tomega_{n-1} = (\bfs_{(0,n)}s_{n-1})^{n-1}  \quad \text{and} \quad 
		\tomega_{n}=	\tau s_n\bfs_{(n-1,n)}\bfs_{(n-2,n)}\cdots \bfs_{(1,n)}, 
		\]
		respectively, where $\tau$ is the diagram automorphism of $C_{n}^{(1)}$ (resp. $D_{n+1}^{(2)}$) sending $i$ to $n-i$ for $i \in I$. 
\end{remark}

Now, let us define the root vectors in $\Uqg$. We refer the reader to \cite{BN04} for the construction of root vectors $\Xzh{\beta}, \beta\in \triangle$ (i.e., $E_\beta$'s defined therein). In particular, the real root vectors $\Xzh{k\td_i\delta \pm \alpha_i}$ are described explicitly by 
\[\Xzh{k\td_i\delta + \alpha_i} = T_{\tomega_i}^{-k}X_i^+ \;\; (k\geq 0), \quad  \Xzh{k\td_i\delta - \alpha_i} = T_{\tomega_i}^{k}T_i\inv X_i^+  \;\; (k> 0). \]
Then $\Xzh{k\td_i\delta \pm \alpha_i}\in U^+$. The imaginary root vectors are defined by 
\begin{equation}\label{Equ:psi}
	\tpsi_{i,k\td_i}=\Xzh{k\td_i\delta-\alpha_i}\Xzh{i}- q_i^{-2}\Xzh{i}\Xzh{k\td_i\delta-\alpha_i} \quad (k>0)
\end{equation}
and define the elements $\Xzh{i, k\td_i\delta}$ by the following formal series in $z$:
\begin{equation}\label{formalseries}
	\mathrm{exp}\left((q_i-q_i\inv)\sum_{k\geq 1}\Xzh{i, k\td_i\delta}z^k\right) = 1+ \sum_{k\geq 1}(q_i-q_i\inv)\tpsi_{i,k\td_i}z^k.
\end{equation}

Under the isomorphism of two presentations of $U_q(\frakg)$, the generators $\psi_{i, k\td_i}^+$ and the imaginary root vectors $\tpsi_{i,k\td_i}$ are related (cf. \cite{Beck94,Dami12}), more precisely, 
	for $k>0$ and $i\in I_0$, we have 
	\begin{equation}
			\psi_{i, k\td_i}^+ = o(i)^k(q_i-q_i\inv)C^{-k\td_i/2}k_i\tpsi_{i,k\td_i}
	\end{equation}
	where $o:I_0 \rightarrow \{\pm1\}$ is a map such that $o(i)=-o(j)$ whenever 
	\begin{enumerate}[\rm i)]
		\item $a_{ij} \leq 0$ implies that $ o(i)o(j)=-1 $,
		\item in the twisted cases different from $A_{2n}^{(2)}$, if $a_{ij}=-2$ then $o(i)=1$. 
	\end{enumerate}

Note that $o(n)=1$ in the type $D_{n+1}^{(2)}$ as $a_{n,n-1}=-2$. Thus, we can deduce that the map $o: I_0 \rightarrow \{\pm 1 \}$ is uniquely determined in the type $D_{n+1}^{(2)}$.

\vspace{.3cm}
In Lemma \ref{Lemma:l-weight}, the scalars $a$ and $b$ can be described by the root vectors according to the above relations, which will become more computable in our case. Let $v\in L(\bsf)$ be a nonzero $\ell$-weight vector of $\bsf$. Since $C^{1/2}$ acts trivially on $L(\bsf)$ and $k_i$ commutes with $\tpsi_{i,k\td_i}$, it implies that 
\begin{align}\label{Equ:XX+}
	\Xzh{2\td_i\delta-\alpha_i}.v = o(i)a\Xzh{\td_i\delta-\alpha_i}.v \quad \text{and}\quad 
	\tpsi_{i,\td_i}.v = o(i)\frac{b}{q_i-q_i\inv}v. 
\end{align}

\begin{lemma}
 For any $i \in I_0$, $k\in \bbZ_{>0}$, we have the root vectors $\Xzh{k\td_i\delta-\alpha_i}$ in $U_q(\AffX)$ have the following relations:
	\[\Xzh{(k+1)\td_i\delta-\alpha_i}= \begin{dcases}
		\frac{1}{[3]_n^!} [\Xzh{k\delta-\alpha_n}, [\Xzh{\delta-\alpha_n}, \Xzh{n}]_q] & \text{if}\;  (X_N^{(r)}, i) = (A_{2n}^{(2)}, n), \\
		\frac{1}{[2]_i}[\Xzh{k\td_i\delta-\alpha_i}, [\Xzh{\td_i\delta-\alpha_i}, \Xzh{i}]_{q_i^2}] & \text{otherwise}.
	\end{dcases}
	\]
\end{lemma}
\begin{proof}
	We may use the following relations  \cite[Proposition 2.2.4, Corollary 3.2.4]{Dami98} (cf. \cite{Beck94}): for $k\in \bbZ_{>0}$, 
	 \begin{equation}\label{Equ1}
	 	\begin{dcases}
	 	[\Xzh{k\delta-\alpha_n},\Xzh{n, \delta}] = [3]_n^! \Xzh{(k+1)\delta-\alpha_n} & \text{if}\;  (X_N^{(r)}, i) = (A_{2n}^{(2)}, n), \\
	 		[\Xzh{k\td_i\delta-\alpha_i},\Xzh{i, \td_i\delta}] = [2]_i \Xzh{(k+1)\td_i\delta-\alpha_i}& \text{otherwise}.
	 	\end{dcases}
	 \end{equation}
	 Note that $\Xzh{i, \td_i\delta}\in \Uqg$ is defined by the formal series (\ref{formalseries}). Since $\tpsi_{i,k\td_i}= [\Xzh{k\td_i\delta-\alpha_i}, \Xzh{i}]_{q_i^2}$ in $(\ref{Equ:psi})$, by comparing the coefficients of $z$ in (\ref{formalseries}), we can get 
	 \begin{equation}\label{Equ2}
	 	\Xzh{i, \td_i\delta} = \tpsi_{i,\td_i} = [\Xzh{\td_i\delta-\alpha_i}, \Xzh{i}]_{q_i^2},
	 \end{equation}
which implies the lemma by (\ref{Equ1}) and (\ref{Equ2}).
\end{proof}

Let $\alpha \in Q_+$. We introduce the \emph{height} $\mathrm{ht}\alpha$ of $\alpha$ as $\mathrm{ht}\alpha = \sum_{i\in I} m_i$ if $\alpha = \sum_{i\in I}m_i\alpha_i$. Define a subset $Q_+(\alpha)$ of $Q_+$ as follows:
\[Q_+(\alpha) = \{\beta \in Q_+\; |\; \mathrm{ht}\alpha - \mathrm{ht}\beta = 1, \alpha-\beta \neq \alpha_0 \}. \]
Let $U^+(\alpha)$ be the subspace of $U^+_\alpha$ defined as $U^+(\alpha) = \sum_{\beta\in Q_+(\alpha)}U^+_\beta \Xzh{\alpha-\beta}$. 

\begin{lemma}\label{Lemma:expressions}
	\noindent	$(\rm 1)$ For $i\in I_0$, the root vector $\Xzh{\delta-\alpha_i}$ in $U_q(A_{n-1}^{(1)})$ has the following form: 
	\begin{flalign*}
	\hspace{1cm}&	\Xzh{\delta-\alpha_i} \equiv (-q^{-1})^{n-2}(\Xzh{i+1}\cdots\Xzh{n-1})(\Xzh{i-1}\cdots\Xzh{2} \Xzh{1})\Xzh{0}\quad  (\mathrm{mod}\; U^+(\delta-\alpha_i)). &
	\end{flalign*}
	
	\noindent	$(\rm 2)$ In $U_q(C_n^{(1)})$,
\begin{alignat*}{3}
	&\Xzh{\delta-\alpha_{n-1}} &&\equiv q^{-n}(\Xzh{n}\Xzh{n-2}\cdots\Xzh{1})(\Xzh{n-1}\Xzh{n-2}\cdots \Xzh{1})\Xzh{0}\quad  (\mathrm{mod} \; U^+(\delta-\alpha_{n-1})), \\ 
	&\Xzh{\delta-\alpha_n} &&\equiv \Big(\frac{q\inv}{[2]_{1}}\Big)^{n-1}(\Xzh{n-1})^2(\Xzh{n-2})^2\cdots (\Xzh{1})^2\Xzh{0}\quad  (\mathrm{mod} \; U^+(\delta-\alpha_n)). 
\end{alignat*}

	\noindent	$(\rm 3)$ In $U_q(\AffAt)$, 
	\begin{flalign*}
	\hspace{1cm}&	\Xzh{\delta-\alpha_n} \equiv q^{-2n}(\Xzh{n-1}\Xzh{n-2}\cdots \Xzh{1})(\Xzh{n}\Xzh{n-1}\cdots \Xzh{1})\Xzh{0} \quad  (\mathrm{mod} \; U^+(\delta-\alpha_n)). &
	\end{flalign*}

	\noindent	$(\rm 4)$ In $U_q(\AffDt)$, 
\begin{flalign*}
\hspace{1cm}&	\Xzh{\delta-\alpha_n} \equiv q^{-2n+2}\Xzh{n-1}\Xzh{n-2}\cdots \Xzh{1}\Xzh{0} \quad  (\mathrm{mod} \; U^+(\delta-\alpha_n)). &
\end{flalign*}
\end{lemma}
\begin{proof}
	Thanks to the reduced expressions of $\tomega_i$ in Remark \ref{Remark:Reducedexpression}, the lemma can be deduced directly by the definition. One can refer to \cite[Lemma 4.7]{JKP21} for the assertion (1). To see the remaining assertions we define the operators $\calD^{(1)}_i$ and $\calD^{(2)}_i$ of $\Uqg$ for $i \in I_0$ by $\calD^{(1)}_{i}(X) = [X, \Xzh{i}]_{q_i}$ and  $\calD^{(2)}_i(X) = [[X, \Xzh{i}], \Xzh{i}]_q$ for any $X\in \Uqg$, respectively. In the type $C_n^{(1)}$, we note that $T_\tau \calD^{(s)}_i = \calD^{(s)}_{n-i}T_\tau$ and $T_j\calD^{(s)}_i = \calD^{(s)}_iT_j$ for $|i-j|>1$ and $s = 1, 2$. Denote $T_{\bfs_{(i,j)}}=T_{(i,j)}$ for simplicity. Then we have 
	\[T_{(0,n)}T_{n-1}\calD_n^{(s)} = \calD_n^{(s)} T_{(0,n)}T_{n-1}, \]
	due to $\bfs_{(0,n)}s_{n-1}\alpha_n = \alpha_n$. Moreover, for any $0 < i < n-1$, we have $(\bfs_{(0,n)}s_{n-1})^{i-1}\alpha_1 = \alpha_i$ and $\bfs_{(i+1, n-1)}\bfs_{(i, n-2)}\alpha_{n-2} = \alpha_i$ by using Lemma  \ref{Lemma:Weylgroupelements}, thus we get  
	\begin{align*}
		&(T_{(0,n)}T_{n-1})^iT_{(0,n-2)}X_{n-1}^+ \\
		&\hspace{2cm} = (T_{(0,n)}T_{n-1})^{i-1}T_{(0,n)}T_{(0,n-3)}X_{n-2}^+ \\
		&\hspace{2cm} =(T_{(0,n)}T_{n-1})^{i-1}T_{(0,n-2)}T_{(0,n-3)}T_{n-1}X_{n-2}^+ \\
		&\hspace{2cm} =-(T_{(0,n)}T_{n-1})^{i-1}[T_{(0,n-2)}\Xzh{n-1}, \Xzh{1}]_{q_{n-1}} \\
		&\hspace{2cm} =-\calD_i^{(1)}(T_{(0,n)}T_{n-1})^{i-1}T_{(0,n-2)}\Xzh{n-1},
	\end{align*}
	and 
	\begin{align*}
		T_{(i+1, n)}T_{(i, n-1)}\Xzh{n} &= T_{(i+1, n-1)}T_{(i,n-2)}T_nT_{n-1}\Xzh{n} \\
		& = T_{(i+1, n-1)}T_{(i,n-2)}T_{n-1}\inv\Xzh{n}\\
		&=\frac{1}{[2]_{n-1}}[[T_{(i+1,n-1)}\Xzh{n}, \Xzh{i}], \Xzh{i}]_q \\
		&= \frac{1}{[2]_1}\calD^{(2)}_iT_{(i+1,n-1)}\Xzh{n} . 
	\end{align*}
Finally, the definition of the root vectors and Remark \ref{Remark:Reducedexpression} imply that 
\begin{align*}
	\Xzh{\delta-\alpha_{n-1}} &= (T_{(0,n)}T_{n-1})^{n-2}T_{(0,n)}\Xzh{n-1} \\
	& = - (T_{(0,n)}T_{n-1})^{n-2}T_{(0,n-2)}\calD_n^{(1)}\Xzh{n-1} \\
	& = - \calD_n^{(1)}(T_{(0,n)}T_{n-1})^{n-2}T_{(0,n-2)}\Xzh{n-1} \\
	& =  \calD_n^{(1)}\calD_{n-2}^{(1)}(T_{(0,n)}T_{n-1})^{n-3}T_{(0,n-2)}\Xzh{n-1} \\
	& \hspace{1cm}\cdots \cdots \\
	& = (-1)^{n-1}\calD_n^{(1)}\calD_{n-2}^{(1)}\cdots \calD_{1}^{(1)}T_{(0,n-2)}\Xzh{n-1} \\
	&= \calD_n^{(1)}\calD_{n-2}^{(1)}\cdots \calD_{1}^{(1)} \calD_{n-1}^{(1)}\cdots \calD_{1, q}^{(1)}\Xzh{0}, 
\end{align*}
where $\calD_{i, q}^{(1)}(X) = [X, \Xzh{i}]_{q}$ for $X \in U$, and 
 \begin{align*}
 	\Xzh{\delta-\alpha_n} 
 	&= T_\tau T_n T_{(n-1,n)}\cdots T_{(2,n)}T_{(1,n-1)}\Xzh{n} \\
 	&= \frac{1}{[2]_1}T_\tau T_n T_{(n-1,n)}\cdots T_{(3,n)}\calD^{(2)}_{1}T_{(2,n-1)}\Xzh{n} \\
 	&=\frac{1}{[2]_1}\calD^{(2)}_{n-1}T_\tau T_n T_{(n-1,n)}\cdots T_{(3,n)}T_{(2,n-1)}\Xzh{n}\\
 	& \hspace{1cm}\cdots \cdots \\
 	&=\Big(\frac{1}{[2]_1}\Big)^{n-1}\calD^{(2)}_{n-1}\calD^{(2)}_{n-2}\cdots\calD^{(2)}_1\Xzh{0},
 \end{align*}
which deduce the assertion (2). Similarly, we can prove that 
\begin{align}
\label{AtXn}	&\Xzh{\delta-\alpha_n}  = -\calD^{(1)}_{n-1}\cdots\calD^{(1)}_{1}\calD_{n, q_1}^{(1)}\calD_{n-1}^{(1)}\cdots\calD_2^{(1)}\calD_{1, q_0}^{(1)}(\Xzh{0}) \quad \text{in}\quad U_q(\AffAt), \\
	&\Xzh{\delta-\alpha_n}  = (-1)^{n-1}\calD^{(1)}_{n-1}\cdots\calD^{(1)}_{1}(\Xzh{0}) \quad \text{in}\quad U_q(\AffDt).
\end{align}
which imply (3) and (4). 
\end{proof}

\begin{remark}
In order to simplify computations in the following theorem for the type $\AffAt (n \geq 2)$, we actually only need the two terms of  $\Xzh{\delta-\alpha_n}$: 
	\begin{align}\label{Aexpression}
		q^{-2n}(\Xzh{n-1}\Xzh{n-2}\cdots \Xzh{1})(\Xzh{n}\Xzh{n-1}\cdots \Xzh{1})\Xzh{0} - q^{-2n+1}(\Xzh{n-1}\cdots\Xzh{1})^2\Xzh{0}\Xzh{n},
	\end{align}
	which can be deduced directly by the formula (\ref{AtXn}).
\end{remark}

Now we compute the highest $\ell$-weights of the $q$-oscillator representations defined in Subsection \ref{Subsection:q-oscillator representations}.

\begin{theorem}
$(\rm 1)$ 
	Fix $0 < s < n$ and $l \in \bbZ$. The $U_q(A_{n-1}^{(1)})$-module $\calW_{s}^{(l)}$ has the highest $\ell$-weight $\bsf = (f_i(z))_{i \in I_0}$ given as follows:
			\[
	f_i(z)=
		\frac{c_{i,l}+u}{1+ c_{i,l}u} \quad \text{with}\quad c_{i,l}=
		\begin{dcases}
			q^{\delta_{i, s-1}l - \delta_{i, s}(l+1)} & \text{if}\quad l\geq 0  \\
			q^{\delta_{i, s}(l-1) - \delta_{i, s+1}l} & \text{if}\quad l< 0
		\end{dcases}
	\]
 for $1 \leq i \leq n-1$, where $u = o(s)(-q\inv)^nz$, 
 
\noindent $(\rm 2)$ The highest $\ell$-weight of the $U_q(C_{n}^{(1)})$-module $\calW^+$  (resp. $\calW^-$) is given as follows: 
\[(1, \cdots, 1,  \frac{q^{-1/2}+u}{1+ q^{-1/2}u }) \quad (\text{resp.} \quad  (1, \cdots, 1, \frac{q^{1/2}+u}{1+q^{1/2}u}, \frac{q^{-3/2}+u}{1+q^{-3/2}u}))\]
where $u = o(n)q^{-n-1}z$. 

\noindent $(\rm 3)$ The highest $\ell$-weight of the $U_q(\AffAt)$ (resp.  $U_q(D_{n+1}^{(2)})$)-module $\calW$  is given by 
\[(1, \cdots, 1, \frac{\imath q_n\inv +u}{1 + \imath q_n\inv u}) \]
where $u = o(n)\imath \tau_q q^{-2n-1}z$ (resp. $ u = q^{-2n}z $). 

\end{theorem}
\begin{proof}
The proof of the first assertion can be found in \cite[Theorem 4.10]{KL2022}. For (2), we have verified in Proposition \ref{Prop:Highestweights} that $v^+=\cket{{\bf 0}}$ and $v^-=\cket{\bfe_n}$ are  highest $\ell$-weight vectors of $U_q(C_n^{(1)})$-modules $\calW^+$ and $\calW^-$ respectively. Therefore, it follows from Lemma \ref{Lemma:l-weight} and the formulae (\ref{Equ:XX+}) that we only need to compute the actions of $\Xzh{2\td_i\delta-\alpha_i}$ and $\tpsi_{i,\td_i}$ on $v^{\pm}$. 

Note that $\Xzh{j}.v^\pm = 0$ for all $j \in I_0$.  By using Lemma \ref{Lemma:expressions} we have 
\[\Xzh{\delta-\alpha_n}.v^+ = \frac{q^{-n+1}}{[2]_1}\cket{2\bfe_n}, \quad  \Xzh{\delta-\alpha_n}.v^- = \frac{q^{-n+1}}{[2]_1}\cket{3\bfe_n},\]
and then 
\[\tpsi_{n,1}.v^+ = \frac{q^{-n-1}}{[2]_1}v^+, \quad \tpsi_{n,1}.v^- = \frac{q^{-n-1}}{[2]_1}[3]_1v^-. \]
Therefore, we have 
\begin{align*}
	\Xzh{2\delta-\alpha_n}.v^+ & =\frac{1}{[2]}[\Xzh{\delta-\alpha_n}, \tpsi_{n,1}].v^+=\frac{1}{[2]}(\frac{q^{-n-1}}{[2]_{1}}\Xzh{\delta-\alpha_n}.v^+- \frac{q^{-n+1}}{[2]_{1}}\tpsi_{n,1}.\cket{2\bfe_n}), \\
	&= \frac{q^{-n+1}}{[2][2]_{1}}(q^{-2}+1-q^{-2}\frac{[4]_1[3]_1}{[2]_1})\Xzh{\delta-\alpha_n}.v^+\\
	&=-q^{-n-3/2}\Xzh{\delta-\alpha_n}.v^+
\end{align*}
and 
\begin{align*}
		\Xzh{2\delta-\alpha_n}.v^- & = \frac{1}{[2]}[\Xzh{\delta-\alpha_n}, \tpsi_{n,1}].v^- = \frac{q^{-n+1}}{[2][2]_{1}}(q^{-2}[3]_1\Xzh{\delta-\alpha_n}.v^-- \tpsi_{n,1}.\cket{3\bfe_n})\\
		& = \frac{q^{-n+1}}{[2][2]_{1}}(q^{-2}[3]_1+[3]_1 - q^{-2}\frac{[5]_1[4]_1}{[2]_1})\Xzh{\delta-\alpha_n}.v^- \\
		&= - q^{-n-5/2}\Xzh{\delta-\alpha_n}.v^-. 
\end{align*}
On the other hand, 
\[\Xzh{\delta-\alpha_{n-1}}.v^- = -q^{-n}\cket{\bfe_{n-1}}, \quad  \tpsi_{n-1,1}.v^- = q^{-n-1}v^-,\]
and then 
\begin{align*}
	\Xzh{2\delta-\alpha_{n-1}}.v^- & = \frac{1}{[2]_{n-1}}[\Xzh{\delta-\alpha_{n-1}}, \tpsi_{n-1,1}].v^- =
 \frac{q^{-n}}{[2]_{n-1}}(q\inv\Xzh{\delta-\alpha_{n-1}}.v^-+ \tpsi_{n-1,1}.\cket{\bfe_{n-1}})\\
	& = \frac{q^{-n}}{[2]_{n-1}}(q\inv+1)\Xzh{\delta-\alpha_{n-1}}.v^- \\
	&=  q^{-n-1/2}\Xzh{\delta-\alpha_{n-1}}.v^-. 
\end{align*}
In other cases, we can check that $\Xzh{\delta-\alpha_i}.v = 0$, then $\tpsi_{i,1}.v = 0 $ and $\Xzh{2\delta-\alpha_i}.v = 0$. Thus, we get (2) as desired.

To get (3), let $v:= \cket{0}$. We first focus on the type $D_{n+1}^{(2)}$.  By Lemma \ref{Lemma:expressions}(4) we have 
\[\Xzh{\delta-\alpha_n}.v =q^{-2n+2}\cket{2\bfe_n}, \quad \tpsi_{n,1}.v = -q^{-2}\Xzh{n}\Xzh{\delta-\alpha_n}.v = -\imath \tau_\nu q^{-2n}v,  \]
and then 
\begin{align*}
	\Xzh{2\delta-\alpha_n}.v & =\frac{1}{[2]}[\Xzh{\delta-\alpha_n}, \tpsi_{n,1}].v=\frac{1}{[2]}(-\imath\tau_\nu q^{-2n}\Xzh{\delta-\alpha_n}.v- q^{-2n+2}\tpsi_{n,1}.\cket{\bfe_n}), \\
	&= \frac{q^{-2n+2}}{[2]}(-\imath\tau_\nu q^{-2}-\imath\tau_\nu+\imath\tau_\nu [2]_\nu q^{-2})\Xzh{\delta-\alpha_n}.v\\
	&=-\imath q^{-2n-1}\Xzh{\delta-\alpha_n}.v.
\end{align*}
For $i \neq n$,  we have $\Xzh{\td_i\delta-\alpha_i}.v = 0$, then $\tpsi_{i,\td_i}.v = 0 $ and $\Xzh{2\td_i\delta-\alpha_i}.v = 0$. Thus, the assertion (3) for  the type $D_{n+1}^{(2)}$ is proved. In the type $A_{2n}^{(2)}$, Lemma \ref{Lemma:expressions}(3) yields 
\[\Xzh{\delta-\alpha_n}.v =q^{-2n}\imath \tau_q\cket{\bfe_n}, \quad \tpsi_{n,1}.v = -q^{-1}\Xzh{n}\Xzh{\delta-\alpha_n}.v =  \tau_q^2 q^{-2n-1}v.  \]
Note that all terms in the expression of $\Xzh{\delta-\alpha_n}$ vanish on the vector $\cket{\bfe_n}$ except for the two terms in (\ref{Aexpression}). We can compute the following action by using  (\ref{Aexpression}): 
\[\Xzh{\delta-\alpha_n}.\cket{\bfe_n} = \imath \tau_q q^{-2n-1}\cket{2\bfe_n}. \]
Therefore, 
\begin{align*}
	\Xzh{2\delta-\alpha_n}.v & =\frac{1}{[3]_n^!}[\Xzh{\delta-\alpha_n}, \tpsi_{n,1}].v=\frac{1}{[3]_n^!}(\tau_q^2 q^{-2n-1}\Xzh{\delta-\alpha_n}.v- \imath \tau_q q^{-2n}\tpsi_{n,1}.\cket{\bfe_n}), \\
	&= \frac{q^{-2n}\tau_q}{[3]_n^!}(q\inv +1 - q^{-2}[2])\Xzh{\delta-\alpha_n}.v\\
	&=q^{-2n-3/2}\tau_q\Xzh{\delta-\alpha_n}.v.
\end{align*}
Then $a=o(n)\tau_q q^{-2n-3/2}$ and $b=o(n)\tau_q q^{-2n}(q^{-1/2}+q^{-3/2})$. The assertion (3) for the type $A_{2n}^{(2)}$ follows from Lemma \ref{Lemma:l-weight} and Corollary \ref{Highestweight}(3). 
\end{proof}

\section*{Acknowledgements}
This paper was partially supported by the NSF of China (11931009, 12161141001, 12171132 and 11771410) and Innovation Program for Quantum Science and Technology (2021ZD0302902). The author would like to thank Prof. Yun Gao and Prof. Hongjia Chen for discussions and their encouragements. The author would also like to thank the anonymous referee for helpful suggestions which improved the exposition of this paper.

     \appendix
     \section{Proof of Theorem \ref{mainthm}}\label{Appendix}
     \subsection{Proof of Theorem \ref{mainthm}}
     
   	For generalized Cartan matrices of finite types, the corresponding system of equations $ (\ref{mainequ}) $ has been solved in \cite{CGLW}, the prescription used in this appendix is parallel with the one there.

     Given an affine Cartan matrix $A$. Fix one $i \in I$,  denote $\calA_{x_i}$ the subalgebra of $\calA$ generated by all $x_j^{\pm 1}$ for $j \neq i$. Then there is a natural isomorphism $ \calA \cong \calA_{x_i}[x_i^{\pm 1}].$  Suppose that $(\phi_i)_{i \in I}$ is any solution to the system of equations (\ref{mainequ}).  
     
     We have the following two crucial lemmas. 
     
     \begin{lemma} \label{lemformu}
     	Any $\phi \in \calA$ satisfying  $\zeta_i(\phi)-\phi = \{y_i\}_i $ has the form $ \phi = \betazh_iy_i  + \phi_{0} + \betafu_iy_i\inv$, where $\phi_{0} \in \calA_{x_i}$, $\betazh_i = -q_i(q_i-q_i\inv)^{-2}$ and $ \betafu_i= -q_i\inv(q_i-q_i\inv)^{-2} $. 
     \end{lemma}
     \begin{proof}
     	Let $\phi = \sum_k \phi_k x_i^k$ with $\phi_k \in \calA_{x_i}$. Then $\zeta_i(\phi)-\phi = \{y_i\}_i $ implies that 
     	\[\sum_{k} (q_i^{-k}-1)\phi_k x_i^k = \frac{1}{q_i- q_i\inv} x_i^2 \prod_{j\neq i}x_j^{a_{ji}} - \frac{1}{q_i- q_i\inv} x_i^{-2} \prod_{j\neq i}x_j^{-a_{ji}}. \]
     	Hence $\phi_k$ is zero unless $k = 0, \pm 2$, and 
     	\[\phi_2 = \betazh_i \prod_{j\neq i}x_j^{a_{ji}}, \quad \phi_{-2} = \betafu_i \prod_{j\neq i}x_j^{-a_{ji}}.   \]
     	So the lemma is proved.
     \end{proof}
     Therefore,  we may always assume that $\phi_i$ in the system of equations (\ref{mainequ}) satisfies  $\phi_i = \betazh_iy_i  + \phi_{i, 0} + \betafu_iy_i\inv$ where  $\phi_{i, 0} \in \calA_{x_i}$. 
     
     Note that any pair $(\phi_i, \phi_j)$ is $(i, j)$-shiftable. This condition can further restrict the choices of $\phi_{i,0}$ and $\phi_{j,0}$ when the nodes $i$ and $j$ are not connected in the Dynkin diagram of $A$, namely, $a_{ij} = 0$. More precisely, we have 
     \begin{lemma}\label{disconnected}
     	If $a_{ij} = 0 $, then both $\phi_{i, 0}$ and $ \phi_{j, 0} $ lie in $\calA_{x_i} \cap \calA_{x_j}$. 
     \end{lemma}
     \begin{proof}
     	Since $a_{ij}= 0$, then $a_{ji} = 0$, we have $y_i \in \calA_{x_j}$ and $y_j \in \calA_{x_i}$. If $\phi_{i, 0} = 0$, there is nothing to do. Assume that $\phi_{i,0}$ is not zero. We rewrite $\phi_{i, 0}$ uniquely as the Laurent polynomial in $x_j$, i.e., a unique form in $\calA_{x_j}[x_j^{\pm 1}]$. Take the nonzero term in this form of $\phi_{i,0}$ such that $x_j$ has the highest (resp. lowest) power, denoted by $\phi_{i, max}$ (resp. $\phi_{i, min}$), then the shiftability of $(\phi_i, \phi_j)$ implies that 
     	\[q_j^m\betazh_j\phi_{i,max}y_j = \betazh_j \phi_{i,max}y_j, \quad q_j^l \betafu_j \phi_{i, min}y_j\inv =\betafu_j \phi_{i, min} y_j\inv,  \]
     	where $m = \Deg_{x_j}\phi_{i,max}$ and $l = \Deg_{x_j}\phi_{i,max}$. Hence $m = 0=l$. Then we can conclude that $\phi_{i,0} \in \calA_{x_i}\cap \calA_{x_j}$ as desired. Similarly, we have $\phi_{j,0} \in \calA_{x_i}\cap \calA_{x_j}$.
     \end{proof}
     
     Let us first focus on the rank-two cases. Fix $i \neq j$ in $I$ and $J = \{i, j \}$. Due to Lemma \ref{disconnected} we may assume that the nodes $i$ and $j$ are connected. 
     Without loss of generality, we set $\lambda = a_{ij}, \mu= a_{ji}$ and $ |\lambda| \geq |\mu| $, then 
     \[ A_J = \left(\begin{array}{@{}cc@{}}
     	2 & \lambda\\
     	\mu & 2
     \end{array}\right)  \text{\; where\ } 1 \leq \lambda \mu \leq 4. \] 
     Then $y_i = x_i^2x_j^\mu$ and $y_j = x_i^\lambda x_j^2$ in this case. Assume that $\varphi_i$ and $\varphi_j$ have the forms as in Lemma \ref{lemformu}, i.e., 
     \[\varphi_l = \betazh_ly_l  + \phi_{l, 0} + \betafu_ly_l\inv, \]
     where $\varphi_{l,0}\in \calA_{x_l},  l \in J$, and let  $\varphi_i$ and $\varphi_j$ satisfy the equality 
     \[	\varphi_i\varphi_j = \zeta_j\inv (\varphi_i) \zeta_i\inv (\varphi_j)  \eqno(\ast) \]
     
     Record the $(x_i, x_j)$-degrees of a monomial $u$ in $\calA$ by a \emph{degree vector}
     \[ \left(\begin{array}{@{}c@{}}
     	\Deg_{x_i} u \\
     	\Deg_{x_j} u
     \end{array}\right) \]
     Then a Laurent polynomial $f$ corresponds to a matrix with each column vector representing for the $(x_i, x_j)$-degrees of certain term of $f$. Moreover, if $\varphi_{i, 0}$ is not zero (resp. $\varphi_{j, 0}$ is not zero), then we use one vector with a parameter $s$ (resp. $t$)
     \[ \left(\begin{array}{@{}c@{}}
     	0 \\
     	s
     \end{array}\right)
     \quad (\text{resp. }
     \left(\begin{array}{@{}c@{}}
     	t \\
     	0
     \end{array}\right)
     )
     \]
     to stand for the possible $(x_i, x_j)$-degrees of $\varphi_{i,0}$ (resp.  $\varphi_{j,0}$). For example, by Lemma \ref{disconnected}, if $a_{ij}=0$, then $s$ and $t$ always equal $0$.
     Therefore, we obtain the following matrix with possible $(x_i, x_j)$-degrees of terms of $\varphi_i\varphi_j$:
     \begin{align*}
     	\left(\begin{array}{@{}ccccccccc@{}}
     		q_i^{\lambda}q_j^\mu & q_i^{t}q_j^\mu & 1 &q_i^{\lambda}q_j^s&q_i^{t}q_j^s&q_i^{-\lambda}q_j^s&1 & q_i^{t}q_j^{-\mu} & q_i^{-\lambda}q_j^{-\mu} \\
     		2+\lambda          &2+t &2-\lambda      & \lambda       & t            & -\lambda& \lambda-2             & t-2            &-\lambda-2 \\
     		\mu+2        & \mu          &  \mu-2      & s+2           & s       &  s-2              &  2-\mu& -\mu  &-2-\mu
     	\end{array}\right)
     \end{align*}
     where the first row of the above matrix is the corresponding \emph{shifted coefficients} in $\zeta_i\inv(\varphi_j)\zeta_j\inv(\varphi_i)$. The terms with {shifted coefficient} $1$ can be cancelled on the left and right hand sides of the equality $ (\ast) $, then we may omit such terms. Therefore, we have the following matrix
     \[ 
     \left(\begin{array}{@{}ccccccc@{}}
     	q_j^{2\mu}& q_i^{t}q_j^\mu  &q_i^{\lambda}q_j^s&q_i^{t}q_j^s&q_i^{-\lambda}q_j^s& q_i^{t}q_j^{-\mu} & q_j^{-2\mu} \\
     	2+\lambda          &2+t      & \lambda       & t            & -\lambda            & t-2            &-\lambda-2 \\
     	\mu+2        & \mu               & s+2           & s       &  s-2              & -\mu  &-2-\mu
     \end{array}\right) \eqno(\rm{M1})
     \]
     One useful statement is that if a shifted coefficient is not $1$, then the corresponding degree vector has to be equal to another one in  the matrix $ (\rm{M1}) $ by the equality $(\ast)$. Therefore we can determine all possible $(x_i, x_j)$-degrees of $\varphi_{i,0}$ and $\varphi_{j,0}$ as follows: 
     
     \vspace{.3cm}
     \begin{center}
     	\begin{tabular}{c|c|c}
     		\hline 
     		Types	& $(\lambda, \mu)$ & Possible values of $s, t$  \\ 
     		\hline
     		$A_2$	& $(-1, -1)$  & $s, t \in \{1, -1 \}$  \\
     		\hline
     		$B_2 (=C_2)$	& $(-2, -1)$ & $\varphi_{i,0}=0, t = \pm 2$ or $s=\pm1, t = 0$  \\
     		\hline
     		$G_2$ & $ (-3, -1) $ & None \\
     		\hline
     		$A_1^{(1)}$	& $(-2, -2)$  & $s= 0 = t$  \\
     		\hline
     		$A_2^{(2)}$	& $(-4, -1)$ & $\varphi_{i, 0} = 0, t = 0$ \\
     		\hline
     	\end{tabular}
     \end{center}
     
     \vspace{.3cm}
     Note that there is no $(x_i, x_j)$-degree vector for the type $G_2$ satisfying the above statement, so does for types $G_2^{(1)}$ and $D_4^{(3)}$. We have obtained all solutions for the type $A_1^{(1)}$ in Example \ref{ExampleA}. In the type $A_2^{(2)}$, we substitute the reduced forms of $\varphi_{i,0}$ and $\varphi_{j,0}$, i.e., $\varphi_{i,0}=0, \varphi_{j,0}\in \bbC^\times$, into  (\ref{mainequ}), and then get 
     \[  \varphi_i = \frac{\imath}{q_i - q_i\inv}\{\imath q^{\frac{1}{2}}x_i^2x_j\inv \}_i, \quad \varphi_j = \{\imath q^{-\frac{1}{2}}x_i^{2}x_j\inv \}_j\{\imath q^{\frac{3}{2}}x_i^{-2}x_j \}_j.  \]
     By our assumption in Section \ref{Sec:Pre}, we have $i=1, j=0$ and $\phi_0 = \varphi_j, \phi_1 = \varphi_i$ for the type $A_2^{(2)}$. 
     
     Let us turn to the higher rank cases. The next result tells us how to ``glue'' the rank-two cases together.
     
     \begin{lemma}\label{Lemma: degree}
     	Let $j \in I$ be a node which connects to the other two distinct nodes $i$ and $l$ in the  Dynkin diagram. Assume that $\phi_{j, 0}\neq 0$ and the pair of integers $(m, t)$ is the $(x_i, x_l)$-degree of any nonzero (monomial) term of $\phi_{j, 0}$. Then we have $mt \leq 0$. 
     \end{lemma}
     \begin{proof}
     	Otherwise, assume that $mt > 0$ and the corresponding nonzero term of $\phi_{j,0}$  is $\phi_{j,0}^{(1)}$. Without loss of generality, we may let $m > 0$ and $t > 0$. Consider the term $\betafu_i y_i\inv \phi_{j,0}^{(1)}$ of $\phi_i \phi_j$ which has the factor $x_i^{m-2}x_j^{-a_{ji}}x_l^{t-a_{li}}$. So we have the shifted coefficient $q_i^mq_j^{-a_{ji}} = q_i^{m-a_{ij}}$ in $\zeta_j\inv\phi_i\zeta_i\inv \phi_j$ is not $1$. However, there is no other term in $\phi_i\phi_j$ whose $(x_i, x_j, x_l)$-degree vector equals $(m-2, -a_{ji}, t-a_{li})$. It is a contradiction. Hence $mt \leq 0$. 
     \end{proof}
     
     The Lemma \ref{Lemma: degree} implies that there is no solution to the system of equations (\ref{mainequ}) for $A$ whose Dynkin diagram contains $D_4$ or $F_4$ as a subdiagram. 
     
     So far, we have ruled out all affine Cartan matrices except that of types  $A_{n}^{(1)} (n \geq 1), C_{n}^{(1)} (n \geq 2), A_{2n}^{(2)} (n \geq 1)$ or $D_{n+1}^{(2)} (n \geq 2)$.  Now we can substitute the reduced forms of $\phi_{i,0}$'s into the system of equations (\ref{mainequ}) to determine the coefficients of the possible terms. Then we obtain all solutions as listed below Theorem \ref{mainthm}. Therefore, Theorem \ref{mainthm} is proved as desired.

  	\bibliographystyle{amsplain}

\end{document}